\documentclass[reqno]{amsart}
\usepackage{amsmath,amsthm,amssymb,IMjournal}
\usepackage{cases}
\usepackage[dvipsnames]{xcolor}
\usepackage{enumerate}
\usepackage{tikz}
\usepackage{blkarray, bigstrut}

\newcommand{\cA}{\mathcal A}

\newcommand{\cG}{\mathcal G}
\newcommand{\bx}{\mathbf x}
\newcommand{\by}{\mathbf y}

\newcommand{\be}{\mathbf e}

\begin{document}
\newtheorem{theorem}{Theorem}[section]
\newtheorem{proposition}[theorem]{Proposition}
\newtheorem{lemma}[theorem]{Lemma}
\newtheorem{corollary}{Corollary}[theorem]
\newtheorem{remark}[theorem]{Remark}
\newtheorem{example}[theorem]{Example}
\newtheorem{observation}[theorem]{Observation}
\newtheorem{definition}[theorem]{Definition}

\numberwithin{equation}{section}

\title{Fiedler vectors with unbalanced sign patterns}

\author{\|Sooyeong |Kim|, Winnipeg,
        \|Steve |Kirkland|, Winnipeg}

\rec {May 14, 2020}

\dedicatory{Cordially dedicated to ...}

\abstract 
In spectral bisection, a Fielder vector is used for partitioning a graph into two connected subgraphs according to its sign pattern. In this article, we investigate graphs having Fiedler vectors with unbalanced sign patterns such that a partition can result in two connected subgraphs that are distinctly different in size. We present a characterization of graphs having a Fiedler vector with exactly one negative component, and discuss some classes of such graphs. We also establish an analogous result for regular graphs with a Fiedler vector with exactly two negative components. In particular, we examine the circumstances under which any Fiedler vector has unbalanced sign pattern according to the number of vertices with minimum degree.
\endabstract
 


\thanks
   This research has been supported by a University of Manitoba Graduate Fellowship (Sooyeong Kim), and by a Discovery Grant from the Natural Sciences and Engineering Research Council of Canada under grant number RGPIN--2019--05408 (Steve Kirkland). The authors would like to thank an anonymous reviewer, whose constructive comments resulted in improvements to the paper.
\endthanks

\section{Introduction and preliminaries}\label{sec1}
When does spectral bisection work well? Recall that spectral bisection is a method to approximately solve the graph partitioning problem: partition a graph $G$ into $k$ subgraphs each of which is similar in size while minimizing the number of edges between each pair of components. There is the result in \cite{Urchel:maximal} about the maximal error in spectral bisection with respect to the minimal cut while partition sizes are the same. In contrast, we shall investigate if spectral bisection is a robust technique by considering the partition sizes. The method uses a so--called Fiedler vector \cite{Fiedler:symmetric} of a graph $G$ so that the edges between two vertices valuated by different signs of the Fiedler vector are cut in order to have the graph $G$ partitioned into two connected subgraphs. The paper \cite{Urschel:bisection} of Urschel and Zikatanov provides a generalization of the work \cite{Fiedler:symmetric} of Miroslav Fiedler with respect to spectral bisection. Specifically, \cite{Urschel:bisection} proves the existence of a Fiedler vector such that two induced subgraphs on the two sets of vertices valuated by non--negative signs and positive signs, respectively, are connected. If all Fielder vectors of a graph $G$ have a sign pattern such that a few vertices are valuated by one sign and possibly $0$, and the others are valuated by the other sign, then spectral bisection will provide an inadequate partition regarding the graph partitioning problem. The present paper examines such graphs and their properties.

Let $G$ be a simple graph of order $n$, that is, $|V(G)|=n$ where $V(G)$ is the vertex set of $G$, and let $H$ be a subgraph of $G$. For $v\in V(H)$, we define $\mathrm{deg}_H(v)$ as the degree of $v$ in $H$. We denote the \textit{minimum degree} and the \textit{vertex connectivity} of $G$ by $\delta(G)$ and $v(G)$, respectively. The \textit{Laplacian matrix} of $G$ is $L(G)=D(G)-A(G)$ where $A(G)$ is the adjacency matrix and $D(G)$ is the diagonal matrix of vertex degrees. The \textit{spectrum} of $L(G)$, $S(L(G))=(\lambda_1(G),\dots,\lambda_n(G))$, is defined as the sequence of eigenvalues of $L(G)$ in non--increasing order. It is well known that $L(G)$ is symmetric and positive semi--definite. In particular, $L(G)\mathbf{1}_n=\mathbf{0}_n$ where $\mathbf{1}_n$ and $\mathbf{0}_n$ are the all ones vector and the zero vector of size $n$, respectively (the subscript will be omitted if no confusion arises). So, $\lambda_n(G)=0$. Similarly, the \textit{spectrum} of $A(G)$, $S(A(G))=(\mu_1(G),\dots,\mu_n(G))$, is defined as the sequence of eigenvalues of $A(G)$ in non--increasing order. Moreover, $\lambda_i(G)$ and $\mu_i(G)$ are written as $\lambda_i$ and $\mu_i$ if $G$ is clear from the context. We use $am(\lambda)$ to denote the algebraic multiplicity of an eigenvalue $\lambda$ of $L(G)$ or $A(G)$. The \textit{algebraic connectivity} $\alpha(G)$ of a graph $G$ is defined as $\lambda_{n-1}(G)$. It is proven in \cite{Fiedler:algebraic} that $\alpha(G)\leq v(G)$ for a non--complete graph $G$. We refer the reader to \cite{Fiedler:algebraic} for more properties of $\alpha(G)$. Since $v(G)\leq \delta(G)$, we have $\alpha(G)\leq\delta(G)$ for a non--complete graph $G$. An eigenvector associated with $\alpha(G)$ is called a \textit{Fiedler vector}. Let $V(G)=\{v_1,\dots,v_n\}$ and $\bx=[x_i]$ be a Fiedler vector of $G$. For $1\leq i\leq n$, a vertex $v_i$ is said to be \textit{valuated} by $x_i$ if $x_i$ is assigned to $v_i$. 

Suppose that $\bx=[x_j]$ is an eigenvector associated to an eigenvalue $\lambda$ of $L(G)$ or $A(G)$. We define $i_\lambda(\mathbf{x})=\mathrm{min}\{|\{x_j|x_j>0\}|,|\{x_j|x_j<0\}|\}$. To distinguish between $L(G)$ and $A(G)$, we define 
\begin{align*}
i_{\lambda}(G):=\underset{\bx\neq0}{\mathrm{min}}\{i_{\lambda}(\mathbf{x})|L(G)\mathbf{x}={\lambda}\mathbf{x}\}\;\;\text{and}\;\;i_{\mu}^*(G):=\underset{\bx\neq0}{\mathrm{min}}\{i_{\mu}(\mathbf{x})|A(G)\mathbf{x}={\mu}\mathbf{x}\}.
\end{align*}
In particular,  $i_{\alpha(G)}(\mathbf{x})$ and $i_{\alpha(G)}(G)$  are denoted as $i(\mathbf{x})$ and $i(G)$, respectively.

We also use some standard terminology and notation in this paper. A vertex $v$ in a connected graph $G$ is a \textit{cut--vertex} if the removal of $v$ and all incident edges results in a disconnected graph. A vertex $v$ in a graph is a \textit{dominating vertex} if $v$ is adjacent to all other vertices. A graph is $r$--\textit{regular} if each vertex of the graph has the same degree $r$. The \textit{complete graph} $K_n$ is the $(n-1)$--regular graph on $n$ vertices. The \textit{empty graph} on $k$ vertices, denoted as $N_k$, consists of $k$ vertices with no edges. The \textit{line graph} of a graph $G$ is the graph whose vertices are the edges of $G$, where two vertices are adjacent if and only if their corresponding edges are incident in $G$. The \textit{complement} $\bar{G}$ of a graph $G$ is a graph with the vertex set $V(G)$ where two vertices are adjacent in $\bar{G}$ if and only if the two vertices are not adjacent in $G$. For two graphs $G_1$ and $G_2$ on disjoint vertex sets, the \textit{disjoint union} $G_1+G_2$ of $G_1$ and $G_2$ is defined as the graph $(V(G_1)\cup V(G_2),E(G_1)\cup E(G_2)))$. For a vertex $v\in V(G)$, $G-v$ is the subgraph of $G$ obtained from $G$ by deleting $v$ and all edges incident with it. The \textit{join} of $G_1$ and $G_2$, denoted as $G_1\vee G_2$, is the graph obtained from $G_1+G_2$ by joining every vertex in $V(G_1)$ to every vertex in $V(G_2)$. Furthermore, $\vee_{i=1}^k G$ is defined as $\underbrace{G\vee\cdots\vee G}_{k\text{ times}}$. It it straightforward to see that $G_1\vee (G_2\vee G_3)=(G_1\vee G_2)\vee G_3$ and $G_1\vee G_2=G_2\vee G_1$.

We introduce the spectral properties of a join of graphs since we use them in several places. Consider two graphs $G_1$ and $G_2$ on disjoint sets of $p$ and $q$ vertices, respectively. Let $S(L(G_1))=(\lambda_1(G_1),\dots,\lambda_p(G_1))$ and $S(L(G_2))=(\lambda_1(G_2),\dots,\lambda_q(G_2))$. It is known (see \cite{Merris:join and spect}) that the (multi--)set of all eigenvalues of $L(G_1\vee G_2)$ is
\begin{align*}
\{0,\lambda_1(G_1)+q,\dots,\lambda_{p-1}(G_1)+q,\lambda_1(G_2)+p,\dots,\lambda_{q-1}(G_2)+p,p+q\}.
\end{align*}
To see this, label the indices of rows and columns of $L(G_1\vee G_2)$ in order of $V(G_1)$ followed by $V(G_2)$. If $\bx$ is an eigenvector orthogonal to $\mathbf{1}_p$ corresponding to $\lambda_i(G_1)$ for $1\leq i\leq p-1$, then $\begin{bmatrix}
\bx^T & \mathbf{0}^T
\end{bmatrix}$ is an eigenvector of $L(G_1\vee G_2)$. Similarly, for an eigenvector $\by$ orthogonal to $\mathbf{1}_q$ corresponding to $\lambda_i(G_2)$ for $1\leq i\leq q-1$, we have $\begin{bmatrix}
\mathbf{0}^T & \by^T 
\end{bmatrix}$ as an eigenvector of $L(G_1\vee G_2)$. Furthermore, $\mathbf{1}_{p+q}$ and $\begin{bmatrix}
-q\mathbf{1}^T & p\mathbf{1}^T
\end{bmatrix}$ are eigenvectors associated with $0$ and $p+q$, respectively.

In Section 2 we find equivalent conditions for $G$ to have $i(G)=1$ (Theorem \ref{equiv:i(G)=1}). In Section 3, all graphs $G$ with $i(\bx)=1$ for all Fiedler vectors $\bx$ are characterized by studying minimum values of $am(\alpha(G))$, according to the number of vertices with minimum degree (Theorem \ref{greatestlowerbounds}). Furthermore, we characterize the graphs for which the sign patterns of all Fielder vectors are extremely unbalanced (Theorem \ref{extreme:i(G)=am(a(G))=1}). In Section 4, threshold graphs with $i(G)=1$ and graphs with three distinct Laplacian eigenvalues and $i(G)=1$ are described. Section 5 provides a characterization of all regular graphs $G$ with $i(G)=2$ by investigating sign patterns of eigenvectors corresponding to the least adjacency eigenvalue of the complement of $G$ (Theorem \ref{Theorem:i(G)=2 for regular}).

Throughout this paper, we assume that all graphs are simple and bold--faced letters are used for vectors.

\section{Characterization of graphs with $i(G)=1$}


\begin{proposition}\label{sign:v=0}
	Let $G$ be a graph of order $n\geq 2$. $G$ is disconnected if and only if $i(G)=0$.
\end{proposition}
\begin{proof}
	Suppose that $G$ is disconnected. Then, $\alpha(G)=0$. So, the all ones vector is a Fiedler vector of $G$. Hence, $i(G)=0$. Conversely, assume that $i(G)=0$. Then there exists a non--negative Fiedler vector $\mathbf{x}$. Since $L(G)\bx=\alpha(G)\bx$, $\mathbf{1}^TL(G)\bx=\alpha(G)\mathbf{1}^T\bx$ and it follows that $\alpha(G)=0$. Hence, $G$ is disconnected.
\end{proof}

For a graph $G$ of order $1$, we have $i(G)=0$, but $G$ is connected. So, if $G$ is a graphs on $n$ vertices where $n\geq 2$, then $i(G)>0$ implies that $G$ is connected.

\begin{lemma}\label{sign:v=1}
	Let $G$ be a non--complete graph of order $n\geq 3$. If $i(G)=1$, then $\alpha(G)=\delta(G)$.
\end{lemma}
\begin{proof}
	Let $\bx$ be a Fiedler vector with $i(\bx)=1$, and we may suppose that $x_1<0$. We have $(L(G)-\alpha(G)I)\bx=0$, and considering the first entry, we find that $(\ell_{11}-\alpha(G))x_1+\sum_{k\neq 1}\ell_{1k}x_k=0$. Since $x_1<0$, $\ell_{1k}\leq 0$ and $x_k\geq 0$ for all $k\neq 1$, it must be the case that $ \ell_{11}\leq\alpha(G)$. Hence $\alpha(G)\geq\delta(G)$, and since $G$ is non--complete, $\alpha(G)\leq\delta(G)$. We deduce that $\alpha(G)=\delta(G)$.
\end{proof}

\begin{example}
	Consider the complete graph $K_n$. Then, $(1,-1,0,\dots,0)^T$ is an eigenvector of $\alpha(K_n)=n$ and by Proposition \ref{sign:v=0}, $i(K_n)=1$. Moreover, $\alpha(G)>\delta(G)=n-1$.
\end{example}

Now, we shall characterize non--complete connected graphs $G$ with $\alpha(G)=\delta(G)$. A characterization of graphs for which $\alpha(G)=v(G)$ appears in \cite{Steve:connectivity}: for a non--complete, connected graph $G$ on $n$ vertices, $\alpha(G)=v(G)$ if and only if there exists a disconnected graph $G_1$ on $n-v(G)$ vertices and a graph $G_2$ on $v(G)$ vertices with $\alpha(G_2)\geq 2v(G)-n$ such that $G=G_1\vee G_2$. Since $\alpha(G)\leq v(G)\leq \delta(G)$, if $\alpha(G)=\delta(G)$, then $\alpha(G)=v(G)=\delta(G)$. So, we begin with a join of a disconnected graph $G_1$ on $n-\delta(G)$ vertices and a graph $G_2$ on $\delta(G)$ vertices with $\alpha(G_2)\geq 2\delta(G)-n$.

\begin{lemma}\label{i=1}
	Let $G$ be a non--complete, connected graph of order $n\geq 3$. Then, $\alpha(G)=\delta(G)$ if and only if $G$ can be expressed as a join of $G_1$ and $G_2$ where the graph $G_1$ on $n-\delta(G)$ vertices has an isolated vertex, and $G_2$ is a graph on $\delta(G)$ vertices, and $\alpha(G_2)\geq 2\delta(G)-n$.
\end{lemma}
\begin{proof}
	 Suppose that $\alpha(G)=\delta(G)$. We will establish the desired conclusion by induction. For order $3$, there is only one graph, $N_1\vee N_2$, that is non--complete and connected; it has the same algebraic connectivity as the minimum degree and has the desired structure. Let $n\geq 4$. Suppose that a graph $G$ of order $n$ with $\alpha(G)=\delta(G)$ is non--complete and connected. Since $\alpha(G)=v(G)=\delta(G)$, $G$ is expressed as $G_1\vee G_2$ where $G_1$ is a disconnected graph of order $n-\delta(G)$, and $G_2$ is a graph of order $\delta(G)$ with $\alpha(G_2)\geq 2\delta(G)-n$. We have $\mathrm{deg}_G(v)\geq \delta(G)$ for $v\in V(G_1)$ and $\mathrm{deg}_G(w)\geq n-\delta(G)$ for $w\in V(G_2)$. If $G_1$ has an isolated vertex, we are done. Suppose that $G_1$ has no isolated vertex. Since $\delta(G_1)>0$, we have $\mathrm{deg}_G(v)>\delta(G)$ for all $v\in V(G_1)$. So, there exists a vertex $w\in V(G_2)$ such that $$\mathrm{deg}_G(w)=\mathrm{deg}_{G_2}(w)+(n-\delta(G))=\delta(G),\;\text{and}\;\mathrm{deg}_{G_2}(w)=\delta(G_2).$$ Since $\mathrm{deg}_{G_2}(w)\geq 0$, we obtain $n-\delta(G)\leq \delta(G)$. 
	
	Suppose that  $n-\delta(G)=\delta(G)$. Then, $\mathrm{deg}_{G_2}(w)=0$ so that $G_2$ has an isolated vertex. Since $G_1$ is disconnected, $\alpha(G_1)=0$. Moreover, $\delta(G)=\frac{n}{2}$. By exchanging the roles of $G_1$ and $G_2$, we obtain the desired description of $G$.
	
	Assume that $n-\delta(G)<\delta(G)$. Note that $\delta(G_2)=2\delta(G)-n$. Since $\alpha(G_2)\geq 2\delta(G)-n$, we obtain $\alpha(G_2)\geq\delta(G_2)$. Suppose that $\delta(G_2)=\delta(G)-1$. Then, we have $\delta(G)=n-1$, which contradicts the non--completeness of $G$. Therefore, $G_2$ is a non--complete, connected graph of order $\delta(G)$ with $\alpha(G_2)=\delta(G_2)$. By induction, there exists a graph $H_1$ of order $\delta(G)-\delta(G_2)$ with an isolated vertex and a graph $H_2$ of order $\delta(G_2)$ such that $G_2=H_1\vee H_2$ and $\alpha(H_2)\geq 2\delta(G_2)-\delta(G)$. Hence, $G=G_1\vee H_1\vee H_2$. Consider $G_1\vee H_2$ of order $n-\delta(G)+\delta(G_2)$. Since $\delta(G_2)=2\delta(G)-n$, the order of $G_1\vee H_2$ is $\delta(G)$. Furthermore, $G_1$ is disconnected so that $\alpha(G_1\vee H_2)$ is either $\delta(G_2)$ or $\alpha(H_2)+n-\delta(G)$. Considering $\alpha(H_2)\geq 2\delta(G_2)-\delta(G)$, it follows that $\alpha(H_2)+n-\delta(G)\geq \delta(G_2)$. So, $\alpha(G_1\vee H_2)=\delta(G_2)=2\delta(G)-n$. Therefore, $G$ can be expressed as a join of $H_1$ and $G_1\vee H_2$.
	
	Conversely, suppose that $G_1$ is a graph of order $n-k$ with an isolated vertex where $1\leq k \leq n-2$, and $G_2$ is a graph of order $k$ with $\alpha(G_2)\geq 2k-n$. Since $\alpha(G_2)+n-k\geq k$, we have $\alpha(G_1\vee G_2)=k$. Let $v$ be an isolated vertex in $G_1$. Then, $\mathrm{deg}_G(v)=k$. So, $\delta(G)\leq k=\alpha(G)$ implies $\delta(G)=\alpha(G)$.
\end{proof}

\begin{remark}\label{order:G_1,G_2}
	{\rm{If $G$ is a non--complete connected graph on $n$ vertices, we have $\delta(G)<n-1$. So, $G_1$ in Lemma \ref{i=1} is of order at least $2$. However, $G_2$ can consist of a single vertex $v$. Then, the vertex $v$ is a cut--vertex of $G$, and also a dominating vertex in $G$.
	
	Considering the fact that $|V(G_1)|\geq 2$ and $G=G_1\vee G_2$, there is no cut--vertex of $G$ in $G_1$. Moreover, if $G_2$ contains a cut--vertex of $G$, $|V(G_2)|=1$. Therefore, if $i(G)=1$, then $G$ has at most one cut--vertex.}}
\end{remark}

\begin{lemma}\label{i=1:a=delta}
	Let $G$ be a non--complete, connected graph of order $n$. Suppose that $G$ can be expressed as a join of $G_1$ and $G_2$ where the graph $G_1$ on $n-\delta(G)$ vertices has an isolated vertex $v$, $G_2$ is a graph on $\delta(G)$ vertices, and $\alpha(G_2)\geq 2\delta(G)-n$. Then, $i(G)=1$.
\end{lemma}
\begin{proof}
	There exists an eigenvector $\mathbf{x}$ corresponding to $\alpha(G)$ where entries corresponding to vertices in $G_1$ except $v$ are all ones, the entry for $v$ is $-(|V(G_1)|-1)$ and zeros elsewhere. Therefore, $i(G)=1$.
\end{proof}

\begin{corollary}
	Let $G$ be a non--complete, connected graph. There exists a cut--vertex $v$ and $i(G)=1$ if and only if $v$ is a dominating vertex that is adjacent to a pendent vertex, that is, $G=(G-v)\vee\{v\}$ where $G-v$ is disconnected.
\end{corollary}
\begin{proof}
	Suppose that $v$ is a cut--vertex in $G$ and that $i(G)=1$. By Remark \ref{order:G_1,G_2}, $G$ is expressed as $G_1\vee G_2$ where $G_1$ contains an isolated vertex $w$ and $G_2=\{v\}$. It is straightforward that $v$ is a dominating vertex and is adjacent to $w$, which is a pendent vertex.
	
	Conversely, suppose that $v\in V(G)$ is a dominating vertex and is adjacent to a pendent vertex $w$. Let $G_1=G-v$ and $G_2=\{v\}$. Then, $w$ is an isolated vertex in $G_1$ and $G=G_1\vee G_2$. By Lemma \ref{i=1:a=delta}, we have the desired result. 
\end{proof}

Thus, the following theorem is obtained by Lemmas \ref{sign:v=1}, \ref{i=1} and \ref{i=1:a=delta}.

\begin{theorem}\label{equiv:i(G)=1}
	Let $G$ be a non--complete, connected graph of order $n$. Then, the following are equivalent:
	\begin{enumerate}
		\item $i(G)=1$,
		\item $\alpha(G)=\delta(G)$,
		\item $G$ can be written as a join of $G_1$ and $G_2$ where the graph $G_1$ on $n-\delta(G)$ vertices has an isolated vertex, $G_2$ is a graph on $\delta(G)$ vertices, and $\alpha(G_2)\geq 2\delta(G)-n$.
	\end{enumerate}
\end{theorem}

\begin{proposition}
	Suppose that $G$ is a connected graph of order $n\geq 3$ and $i(G)\neq 1$. Then, we can construct a graph $G'$ such that $i(G')=1$ and $G$ is an induced subgraph of $G'$ by adding at most two vertices and joining them to some vertices of $G$. In particular, we need only one vertex if $G$ is a join. Otherwise, we need two vertices.
\end{proposition}
\begin{proof}
	Suppose that $G$ can be expressed as a join of two graphs, say $H_1$ of order $n_1$ and $H_2$ of order $n_2$ where $n_1\geq n_2$. Let $G'$ be $(\{v\}+H_1)\vee H_2$ for a new vertex $v$. Then, $\delta(G')=n_2$. Since $\alpha(G')=\mathrm{min}\{n_2,a(H_2)+n_1\}$, we have $\delta(G')=\alpha(G')$, and $i(G')=1$.
	
	Assume that $G$ is not a join of some graphs. Let $H_1=\{v\}+G$ and $H_2=\{w\}$ where $v\neq w$. Consider $G'=H_1\vee H_2$. Since $H_1$ contains an isolated vertex and $\alpha(H_2)=0\geq 2\delta(G')-n$, by Theorem \ref{equiv:i(G)=1}, $i(G')=1$. It remains to show that every graph $H$ obtained from a graph $G$ by adding just one new vertex $v$ and joining it to some vertices does not satisfy $i(H)=1$. Suppose to the contrary that there exists such a graph $H$ with $i(H)=1$. By Theorem \ref{equiv:i(G)=1} and Remark \ref{order:G_1,G_2}, $H$ is expressed as a join of two graphs $G_1$ and $G_2$ where $G_1$ has an isolated vertex and $|V(G_1)|\geq 2$. Suppose that the new vertex $v$ is in $G_1$. Since $|V(G_1)|\geq 2$, a removal of $v$ in $H$ results in the graph $G$ that is a join of some graphs, a contradiction. Hence, $v\in V(G_2)$. Furthermore, $G_2=\{v\}$, for otherwise, $G$ would be written as a join of some graphs. Thus, $G=G_1$ and so $G$ is disconnected. This contradicts  the hypothesis that $G$ is connected. Therefore, we need to add at least two vertices to have a connected graph $G'$ with the desired properties. 
\end{proof}


\section{Algebraic multiplicity of a graph with $i(G)=1$}


Recall that $i(\bx)$ is defined as the minimum number of negative components in $\bx$ or $-\bx$. 

\begin{example}\label{eg:am}
	Let $G_1=K_2+N_1$ and $G_2=N_1\vee N_3$. Since $G_1$ has an isolated vertex and $\alpha(G_2)=2\delta(G_1\vee G_2)-7$, we have $i(G_1\vee G_2)=1$ by Theorem \ref{equiv:i(G)=1}. Furthermore, $\alpha(G_1\vee G_2)=4$ and $am(\alpha(G_1\vee G_2))=3$. Labeling vertices in order of $V(G_1)$ and $V(G_2)$, there are three linearly independent Fiedler vectors corresponding to $\alpha(G_1\vee G_2)$:
	\begin{align*}
	\bx_1^T&=\begin{bmatrix}
	1 & 1 &-2 & 0 & 0 & 0 & 0
	\end{bmatrix};\\
	\bx_2^T&=\begin{bmatrix}
	0 & 0 & 0 & 0 & 1 & -1 & 0
	\end{bmatrix};\\
	\bx_3^T&=\begin{bmatrix}
	0 & 0 & 0 & 0 & 1 & 0 & -1
	\end{bmatrix}.
	\end{align*}
	Therefore, $i(\bx_1+\bx_2)=2$ and $i(\bx_1+\bx_2+\bx_3)=3$.
\end{example}
Let $G$ be a non--complete graph of order $n$ with $i(G)=1$. So, $G$ can be written as $G=G_1\vee G_2$ where the graph $G_1$ on $n-\delta(G)$ vertices contains an isolated vertex, and $G_2$ is a graph on $\delta(G)$ vertices with $\alpha(G_2)\geq 2\delta(G)-n$. We observe from Example \ref{eg:am} that if $\alpha(G_2)=2\delta(G)-n$, then $am(\alpha(G_2))$ must be considered to compute $am(\alpha(G))$. Let $\beta(H)$ denote the number of connected components in a graph $H$. Since the algebraic multiplicity of the eigenvalue $0$ of $G_1$ is $\beta(G_1)$, by considering $G=G_1\vee G_2$, we have
\begin{align}\label{identity:am}
am(\alpha(G))=\begin{cases}
\beta(G_1)-1+am(\alpha(G_2)), & \text{if $\alpha(G_2)=2\delta(G)-n$},\\
\beta(G_1)-1, & \text{if $\alpha(G_2)>2\delta(G)-n$}.
\end{cases}
\end{align}
Moreover, from Example \ref{eg:am}, we see that for a non--complete connected graph $G$ the condition that $i(G)=1$ and $am(\alpha(G))>1$ does not guarantee that $i(\bx)=1$ for every Fiedler vector $\bx$. 

\begin{proposition}\label{am:i(x)}
	Let $G$ be a non--complete graph of order $n$ and $i(G)=1$. Suppose that $G\neq N_3\vee G'$ for any graph $G'$ with $\alpha(G')> 2\delta(G)-n$. Then, $am(\alpha(G))>1$ if and only if there exists a Fiedler vector $\bx$ such that $i(\bx)>1$.
\end{proposition}
\begin{proof}
	Suppose that $am(\alpha(G))>1$. Since $i(G)=1$, there are graphs $G_1$ and $G_2$ such that $G=G_1\vee G_2$ where the graph $G_1$ on $n-\delta(G)$ vertices contains an isolated vertex  and $G_2$ is a graph of order $\delta(G)$ with $\alpha(G_2)\geq 2\delta(G)-n$. Assume that $\alpha(G_2)> 2\delta(G)-n$. From \eqref{identity:am}, we find that there are at least three connected components in $G_1$. Since $G_1\neq N_3$, $|V(G_1)|\geq 4$. Choose two components $H_1$ and $H_2$ of $G_1$ such that $H_1$ and $H_2$ are the smallest and second smallest orders in $G_1$. Then, $H_1=N_1$. Labeling vertices in order of $V(H_1)$, $V(H_2)$, $V(G_1)\backslash(V(H_1)\cup V(H_2))$ and $V(G_2)$, there exists a Fiedler vector $$\bx^T=\begin{bmatrix}
	-1 & -(\frac{|V(G_1)|-|V(H_1)|-|V(H_2)|-1}{|V(H_2)|})\mathbf{1}_{|V(H_2)|}^T & \mathbf{1}_{|V(G_1)|-|V(H_1)|-|V(H_2)|}^T & \mathbf{0}_{|V(G_2)|}^T
	\end{bmatrix}.$$ Then, $\bx$ and $-\bx$ have $|V(H_1)|+|V(H_2)|$ and $|V(G_1)|-|V(H_1)|-|V(H_2)|$ negative components, respectively. It is clear that $|V(H_1)|+|V(H_2)|\geq 2$. Since $G_1\neq N_3$ and $H_1$ and $H_2$ are the components of the smallest and second smallest orders in $G_1$, we have $|V(G_1)|-|V(H_1)|-|V(H_2)|\geq 2$. Therefore, $i(\bx)\geq 2$. 
	
	Suppose that $\alpha(G_2)=2\delta(G)-n$. Let $v$ be an isolated vertex in $G_1$. Then, we have a Fiedler vector $\bx_1=\begin{bmatrix}
	\mathbf{1}_{|V(G_1)|}-|V(G_1)|\be_v\\
	\mathbf{0}_{|V(G_2)|}
	\end{bmatrix}$ where $|V(G_1)|\geq 2$. Choose an eigenvector $\by$ corresponding to $\alpha(G_2)$ such that $\by^T\mathbf{1}=0$ and $i(\by)>0$. Since $\alpha(G_2)=2\delta(G)-n$, $\bx_2=\begin{bmatrix}
	\mathbf{0}_{|V(G_1)|}\\
	\by
	\end{bmatrix}$ is a Fiedler vector of $G$. Then, $i(\bx_1+\bx_2)>1$.
	
	Suppose that there is a Fiedler vector $\bx$ such that $i(\bx)>1$. By hypothesis, there is a Fiedler vector $\bx'$ such that $i(\bx')=1$. Evidently, $\bx'$ is not a scalar multiple of $\bx$, so those two vectors are linearly independent. Hence, $am(\alpha(G))\geq 2$.
\end{proof}

Proposition \ref{am:i(x)} establishes that the condition that $i(G)=1$ and $am(\alpha(G))=1$ forces  any Fiedler vector $\bx$ to have $i(\bx)=1$. Moreover, the set of all graphs $G$ such that $am(\alpha(G))>1$ and $i(\bx)=1$ for all Fiedler vectors $\bx$ is $$\{N_3\vee G'| G' \text{ is a graph with } \alpha(G')> 2\delta(N_3\vee G')-|V(N_3\vee G')|\}.$$ 

We will characterize graphs with $i(G)=1$ and $am(\alpha(G))=1$ by studying the relation between $am(\alpha(G))$ and the number of vertices of degree $\delta(G)$. Before presenting the characterization, lower bounds on $am(\alpha(G))$ will be derived.

\begin{lemma}\label{am.isol2}
	Let $G$ be a non--complete connected graph of order $n$. There are exactly $\ell$ vertices of degree $\delta(G)$ and $i(G)=1$ if and only if for some $k\geq 1$ there are graphs $G_1,\dots,G_k$ satisfying the following conditions:
	\begin{enumerate}
		\item $|V(G_1)|=\cdots=|V(G_{k})|=n-\delta(G)\geq 2$; \label{threeconditions1}
		\item for $i=1,\dots, k$ each $G_i$ contains $\ell_i(\geq 1)$ isolated vertices of degree $\delta(G)$ in $G$, and $\ell=\sum_{j=1}^{k}\ell_j$;\label{threeconditions2}
		\item $G$ is described by one of two cases: \label{threeconditions3}
		\begin{enumerate}
			\item\label{threeconditions31} $G=\vee_{j=1}^{k}G_j$ or 
			\item\label{threeconditions32} $G=(\vee_{j=1}^{k}G_j)\vee G'$ where $G'$ is a graph on $k\delta(G)-(k-1)n$ vertices such that $\mathrm{deg}_G(v)>\delta(G)$ for all $v\in V(G')$ and $\alpha(G')\geq (k+1)\delta(G)-kn$. 
		\end{enumerate}
	\end{enumerate}
\end{lemma}

\begin{proof}
	We will use induction on $\ell$ to prove the necessity of conditions \eqref{threeconditions1},\eqref{threeconditions2} and \eqref{threeconditions3} in order for $G$ to have exactly $\ell$ vertices of degree $\delta(G)$ and $i(G)=1$. The case $\ell=1$ follows immediately from Theorem \ref{equiv:i(G)=1}. Let $\ell\geq 2$. Since $G$ is non--complete and $i(G)=1$, $G$ can be written as a join of two graphs $\hat{G}_1$ and $\hat{G}_2$ where $\hat{G}_1$ is a graph on $n-\delta(G)$ vertices with an isolated vertex and $\hat{G}_2$ is a graph on $\delta(G)$ vertices with $\alpha(\hat{G}_2)\geq 2\delta(G)-n$. The order of $\hat{G}_1$ is more than $1$ by Remark \ref{order:G_1,G_2}. If $\hat{G}_1$ contains $\ell$ isolated vertices, then $\mathrm{deg}_{G}(v)>\delta(G)$ for all $v\in V(\hat{G}_2)$. By choosing $G_1=\hat{G}_1$ and $G'=\hat{G}_2$, we have the desired result with $k=1$, which corresponds to the case \eqref{threeconditions32}. Assume that there are $\ell_1$ isolated vertices in $\hat{G}_1$ where $\ell_1<\ell$. Then, $\hat{G}_2$ contains exactly $\hat{\ell}_2:=\ell-\ell_1$ vertices of degree $\delta(G)$ in $G$. Since $\delta(G)$ is the minimum degree in $G$, the $\hat{\ell}_2$ vertices are also of the minimum degree in $\hat{G}_2$. We have $\delta(\hat{G}_2)=2\delta(G)-n$ from the fact that $G=\hat{G}_1\vee\hat{G}_2$. If $\hat{G}_2$ is complete, then $\delta(\hat{G}_2)=\delta(G)-1$ and so $\delta(G)=n-1$, which  contradicts  the fact that $G$ is non--complete. Hence, $\hat{G}_2$ is a non--complete graph and $\delta(\hat{G}_2)\geq \alpha(\hat{G}_2)$. Since $\delta(\hat{G}_2)=2\delta(G)-n$ and $\alpha(\hat{G}_2)\geq 2\delta(G)-n$, we have $$\delta(\hat{G}_2)=\alpha(\hat{G}_2)=2\delta(G)-n.$$
		
	Assume that $\hat{G}_2$ is disconnected. Then $\alpha(\hat{G}_2)=0$, which yields $\delta(\hat{G}_2)=0$ and $\delta(G)=\frac{n}{2}$. Since $\delta(\hat{G}_2)=0$, the $\hat{\ell}_2$ vertices are the only isolated vertices in $\hat{G}_2$. Moreover, we have $|V(\hat{G}_1)|=|V(\hat{G}_2)|$ since $\delta(G)=\frac{n}{2}$. Setting up $\ell_2=\hat{\ell}_2$, $G_1=\hat{G}_1$, $G_2=\hat{G}_2$, we have the  result with $k=2$, which corresponds to \eqref{threeconditions31}. 
	
	Suppose now that $\hat{G}_2$ is connected. Then, $i(\hat{G}_2)=1$ by Theorem \ref{equiv:i(G)=1}. Since $\hat{\ell}_2<\ell$, by induction, there are graphs $G_2,\dots,G_k$ for some $k\geq 2$ satisfying the conditions:
	\begin{enumerate}[(i)]
	\item $|V(G_2)|=\cdots=|V(G_{k})|=\delta(G)-\delta(\hat{G}_2)=n-\delta(G)\geq 2$;
	\item for $i=2,\dots, k$ each $G_i$ contains $\ell_i(\geq 1)$ isolated vertices of degree $\delta(\hat{G}_2)$ in $\hat{G}_2$ with $\hat{\ell}_2=\sum_{j=2}^{k}\ell_j$; and 
	\item $\hat{G}_2$ is described by one of two cases:
	\begin{enumerate}
		\item $\hat{G}_2=\vee_{j=2}^{k}G_j$ or 
		\item $\hat{G}_2=(\vee_{j=2}^{k}G_j)\vee G'$ where $G'$ is a graph on $(k-1)\delta(\hat{G}_2)-(k-2)|V(\hat{G_2})|$ vertices such that $\mathrm{deg}_{\hat{G}_2}(v)>\delta(\hat{G}_2)$ for all $v\in V(G')$ and $\alpha(G')\geq k\delta(\hat{G}_2)-(k-1)|V(\hat{G_2})|$.
	\end{enumerate}
	\end{enumerate}	
	Clearly, the condition \eqref{threeconditions1} is satisfied. Since the $\hat{\ell}_2$ vertices in $\hat{G}_2$ have degree $\delta(G)$ in $G$, we have $\ell=\ell_1+\hat{\ell}_2=\sum_{j=1}^{k}\ell_j$. So, the condition \eqref{threeconditions2} is shown. Let $G_1=\hat{G}_1$. If $\hat{G}_2=\vee_{j=2}^{k}G_j$, we obtain the case \eqref{threeconditions31}. Suppose that $\hat{G}_2=(\vee_{j=2}^{k}G_j)\vee G'$. Considering the fact that $G=G_1\vee\hat{G}_2$, $\delta(\hat{G}_2)=2\delta(G)-n$ and $|V(\hat{G_2})|=\delta(G)$, it is straightforward to check the remaining conditions in \eqref{threeconditions32}. Therefore, our desired description of $G$ is obtained. 

	For the proof of the converse, suppose that there exists a graph $G$ with $G_1,\dots,G_k$ for some $k\geq 1$ satisfying the conditions \eqref{threeconditions1} and \eqref{threeconditions2} in the statement. For the case \eqref{threeconditions31}, $G$ contains $\ell$ vertices of degree $\delta(G)$ by the condition \eqref{threeconditions2}. Consider the case \eqref{threeconditions32}. Since $\mathrm{deg}_{G}(v)>\delta(G)$ for all $v\in V(G')$, $G$ contains exactly $\ell$ vertices of degree $\delta(G)$. It remains to show $i(G)=1$. Suppose that $G$ is as in case \eqref{threeconditions32}. Note that $\alpha(G')\geq (k+1)\delta(G)-kn$. So, $\alpha(G)$ can be obtained from the eigenvalue $0$ in $G_1$ by computing the spectrum of the join so that $$\alpha(G)=(k-1)(n-\delta(G))+|V(G')|=\delta(G).$$ Therefore, by Theorem \ref{equiv:i(G)=1}, $i(G)=1$. Similarly, for the case \eqref{threeconditions31}, it is straightforward to show that $\alpha(G)=\delta(G)$.
\end{proof}

\begin{remark}\label{uniqueness}{\rm{
	Continuing with the notation and terminology of Lemma \ref{am.isol2}, we have $|V(G')|=k\delta(G)-(k-1)n$ and $|V(G_1)|=n-\delta(G)$. So, $$\alpha(G')\geq (k+1)\delta(G)-kn=|V(G')|-|V(G_1)|.$$
	Furthermore, we observe that the complement $\bar{G}_i$ of each $G_i$ for $i=1,\dots,k$ is connected, so $G_i$ can not be expressed as a join of graphs. Thus, the decomposition of $G$ in terms of joins in Lemma \ref{am.isol2} is unique, (up to the ordering of the graphs). In particular, $k$ is uniquely determined.}}
\end{remark}


\begin{definition}\label{def:elemetary k-join}
	Let $\ell\geq 1$. Graphs $H_1,\dots,H_\ell$ are called \textit{elementary} if
	\begin{enumerate}
		\item $|V(H_1)|=\cdots=|V(H_{\ell})|\geq 2$ and\label{c1}
		\item each $H_i$ for $i=1,\dots,\ell$ contains at least one isolated vertex.\label{c2}
	\end{enumerate} 
	A graph $G$ is said to be an \textit{elementary $k$--join} if $G$ can be written as $G=\vee_{j=1}^{k}G_j$ for some $k\geq 2$ such that $G_1,\dots, G_k$ are elementary. The graphs $G_1,\dots,G_k$ are called \textit{elementary} graphs of $G$.
\end{definition}

\begin{definition}\label{def:combined k-join}
	A graph $G$ on $n$ vertices is said to be a \textit{combined $k$--join} if $G$ can be expressed as $G=(\vee_{j=1}^{k}G_j)\vee G'$ for some $k\geq 1$ such that $G_1,\dots,G_k$ are elementary and $G'$ is a graph on $k\delta(G)-(k-1)n$ vertices such that $\mathrm{deg}_G(v)>\delta(G)$ for all $v\in V(G')$ and $\alpha(G')\geq |V(G')|-|V(G_1)|$. The graphs $G_1,\dots,G_k$ and the graph $G'$ are called the \textit{elementary} graphs and the \textit{combined} graph of $G$, respectively.
\end{definition}

\begin{remark}{\rm{
	If $G$ is an elementary $k$--join, then $k\geq 2$. Otherwise, $G$ would be disconnected. Considering Remark \ref{uniqueness}, an elementary $k$--join $G$ does not imply that $G$ is a combined $k$--join, vice versa.}}
\end{remark}

\begin{definition}\label{def:k-join}
	A graph $G$ is called to be a $k$--join if $G$ is either an elementary $k$--join or a combined $k$--join. 
\end{definition}

\begin{remark}{\rm{
	A $k$--join is not a complete graph.}}
\end{remark}

The following result is straightforward from Lemma \ref{am.isol2}.

\begin{theorem}
	Let $G$ be a non--complete connected graph. Then, $i(G)=1$ if and only if $G$ is a $k$--join.
\end{theorem}

\begin{example}{\rm{
	Consider the Shrikhande graph $G'$ with parameters $(16,6,2,2)$, which is a strongly regular graph, see \cite{Brower:spectral}. By computation, $\alpha(G')=4$ and $am(\alpha(G'))=6$. Let $G_1=K_{11}+\{v\}$. Then $i(G_1\vee G')=1$ and it has only one vertex with the minimum degree, but $am(\alpha(G_1\vee G'))=7$. Moreover, $G_1\vee G'$ is a combined $1$--join.}}
\end{example}

\begin{theorem}\label{identity:amdecomposable}
	Suppose that $G$ is an elementary $k$--join and $G_1,\dots,G_k$ are the elementary graphs of $G$. Then, $am(\alpha(G))=\sum_{i=1}^{k}\beta(G_i)-k$. Assume that $G$ is a combined $k$--join, and $G_1,\dots,G_k$ and $G'$ are the elementary graphs and the combined graph of $G$, respectively. Then, 
	\begin{align*}
	am(\alpha(G))=\begin{cases}
	\sum_{i=1}^{k}\beta(G_i)-k+am(\alpha(G')), & \text{if $\alpha(G')=2\delta(G)-n$},\\
	\sum_{i=1}^{k}\beta(G_i)-k, & \text{if $\alpha(G')>2\delta(G)-n$}
	\end{cases}
	\end{align*}
\end{theorem}
\begin{proof}
	Considering the spectrum of a join of graphs, we immediately obtain the desired result. 
\end{proof}

Let $\cA_\ell$ be the set of all non--complete graphs $G$ with $\ell$ vertices of minimum degree $\delta(G)$ such that $i(G)=1$. For $G\in\cA_\ell$, $G$ is a $k$--join for some $1\leq k \leq \ell$. Note that if $k=1$, then $G$ is a combined $1$--join. In order to attain the minimum of $am(\alpha(G))$ where $G\in\cA_{\ell}$ is a $k$--join, by Theorem \ref{identity:amdecomposable} we only need to consider elementary $k$--joins and combined $k$--joins $G$ where the combined graph $G'$ of $G$ satisfies $\alpha(G')>2\delta(G)-|V(G)|$. Let $\cA_{\ell,k}$ denote the subset of $\cA_\ell$ that consists of elementary $k$--joins and such combined $k$--joins. Define $$m_{\ell,k}:=\mathrm{min}\{am(\alpha(G))|G\in\cA_{\ell,k}\}.$$
 We will investigate $m_{\ell,k}$ and families of graphs attaining $m_{\ell,k}$. Then, the greatest lower bound of $\{am(\alpha(G))|G\in\cA_\ell\}$ will be derived.

Let $G\in\cA_{\ell,k}$ where $1\leq k\leq \ell$. Let $G_1,\dots,G_k$ be the elementary graphs of $G$. For $i=1,\dots,k$, each $G_i$ contains at least one isolated vertex, say $v_i$, so $\beta(G_i)-1$ is the number of connected components in $G_i-v_i$. Since there are $\ell-k$ isolated vertices left in the disjoint union of $G_1-v_1,\dots,G_k-v_k$ by Theorem \ref{identity:amdecomposable}, we have
$$am(\alpha(G))=\ell-k+p(G)$$
where $p(G)$ is the number of components of order more than $1$ in the elementary graphs $G_1,\dots,G_k$ of $G$. Define $$p_{\ell,k}:=\mathrm{min}\{p(G)|G\in\cA_{\ell,k}\}.$$ Therefore, we have
$$m_{\ell,k}=\ell-k+p_{\ell,k}.$$
Then, $m_{\ell,k}$ can be completely determined by considering $3$ cases for $1\leq k\leq\ell$: (i) $k\mid\ell$ where $\ell\geq 2$ and $1\leq k <\ell$, (ii) $k=\ell$ or $k=\ell-1\geq 2$, (iii) $k\nmid \ell$ and $2\leq k\leq \ell-2$.

\begin{lemma}[Case (i)]\label{lemma1:glb}
	Let $G\in\cA_{\ell,k}$ where $\ell\geq 2$ and $1\leq k<\ell$. Suppose that $G_1,\dots,G_k$ are the elementary graphs of $G$. Then, $k\mid\ell$ if and only if $m_{\ell,k}=\ell-k$. In particular, $G_i=N_{a+1}$ for $i=1\dots,k$ where $a\geq 1$ and $\ell=(a+1)k$.
\end{lemma}
\begin{proof}
Note that $k\mid\ell$ if and only if $k\mid\ell-k$. Assume that $\ell-k=ak$ for some $a\geq 1$. By choosing $G_i=N_{a+1}$ for $i=1\dots,k$, we have $p(G)=0$. Hence, $p_{\ell,k}=0$ and $m_{\ell,k}=\ell-k$. Conversely, if $m_{\ell,k}=\ell-k$, then $p_{\ell,k}=0$ and so each $G_i$ must consist of isolated vertices. Since $|V(G_1)|=\cdots=|V(G_k)|\geq 2$, it follows that there is $a\geq 1$ such that $\ell-k=ak$. Furthermore, $G_i=N_{a+1}$ for $i=1,\dots,k$.
\end{proof}

We shall consider an example to see that $p(G)$ depends on how $G_1,\dots,G_k$ consist of isolated vertices.
\begin{example}{\rm{
	Let $G\in\cA_{12,5}$, and let $G_1,\dots,G_5$ be the elementary graphs of $G$. Note that for $i=1,\dots,5$, $G_i$ has at least one isolated vertex. Consider the following configurations of three distributions of $12$ isolated vertices in $G_1,\dots,G_5$:
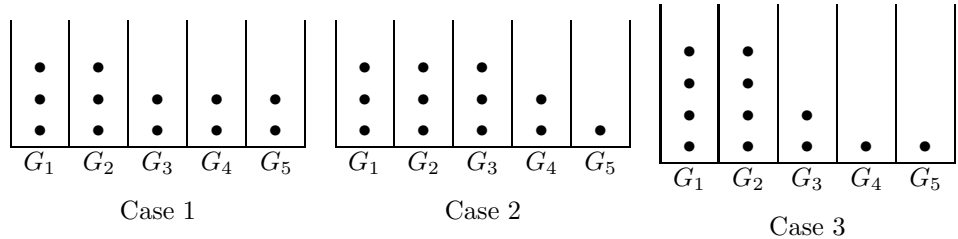
\[
\begin{blockarray}{ccccc}
\begin{block}{|c|c|c|c|c|}
&&&&\\
\bullet & \bullet&&&\\
\bullet& \bullet&\bullet&\bullet& \bullet\\
\bullet& \bullet&\bullet&\bullet& \bullet\\\hline
\end{block}
G_1 & G_2 & G_3 & G_4 & G_5 \\
\noalign{\vspace{1.5ex}}
\BAmulticolumn{5}{c}{%
{\text{Case 1}}
}
\\
\end{blockarray}\;\;\;\;
\begin{blockarray}{ccccc}
\begin{block}{|c|c|c|c|c|}
&&&&\\
\bullet & \bullet&\bullet&&\\
\bullet& \bullet&\bullet&\bullet& \\
\bullet& \bullet&\bullet&\bullet& \bullet\\\hline
\end{block}
G_1 & G_2 & G_3 & G_4 & G_5 \\
\noalign{\vspace{1.5ex}}
\BAmulticolumn{5}{c}{%
	{\text{Case 2}}
}
\\
\end{blockarray}\;\;\;\;
\begin{blockarray}{ccccc}
\begin{block}{|c|c|c|c|c|}
&&&&\\
\bullet&\bullet&&&\\
\bullet & \bullet&&&\\
\bullet& \bullet&\bullet&&\\
\bullet& \bullet&\bullet&\bullet& \bullet\\\hline
\end{block}
G_1 & G_2 & G_3 & G_4 & G_5 \\
\noalign{\vspace{1.5ex}}
\BAmulticolumn{5}{c}{%
	{\text{Case 3}}
}
\\
\end{blockarray}
\vspace{-3mm}
\]
	where for each case, a $\bullet$ indicates an isolated vertex, and the $j^\text{th}$ column describes how many isolated vertices $G_j$ has. Note that for each case, there are no more isolated vertices in $G_j$; $G_j$ may have disconnected components of order more than $1$ under the condition that $|V(G_1)|=\cdots=|V(G_5)|\geq 2$.
	
	Consider Case 1. If $|V(G_i)|=3$ for $i=1,\dots,5$, then $G_3,G_4$ and $G_5$ must have three isolated vertices, a contradiction to $\ell=12$. In order for $G$ to satisfy the condition that it only has $12$ isolated vertices and $|V(G_1)|=\cdots=|V(G_5)|\geq 2$, at least one component of order more than $1$ must be added to each $G_j$. Thus, $p(G)\geq 5$ for Case 1. 
	
	Using the same argument for Case 2, it follows that we also need at least five  components of order more than $1$. Hence, $p(G)\geq 5$ for Case 2. 
	
	For Case 3, we minimally need three components: $K_2$, $K_3$ and $K_3$ in $G_3$, $G_4$ and $G_5$, respectively. Thus, $|V(G_1)|=\cdots=|V(G_5)|\geq 4$ and $p(G)\geq 3$.}}
\end{example}

Let $G\in\cA_{\ell,k}$ where $\ell-k\geq 1$. Suppose that $G_1,\dots,G_k$ are the elementary graphs of $G$, and $v_i$ is an isolated vertex in $G_i$ for $i=1,\dots,k$. Let $c_i(G)\geq 0$ be the number of isolated vertices in $G_i-v_i$ so that $\ell-k=\sum_{i=1}^{k}c_i(G)$. Suppose that $c_{\mathrm{max}}(G):=\mathrm{max}\{c_1(G),\dots,c_k(G)\}$ and $q(G):=\left|\{i|c_i(G)=c_{\mathrm{max}}(G) \text{ for } 1\leq i\leq k\}\right|$. Since $\ell-k\geq 1$, we have $c_{\mathrm{max}}(G),q(G)\geq 1$. If $G$ is clear from the context, then $c_i(G)$ and $c_{\mathrm{max}}(G)$ can be written as $c_i$ and $c_{\mathrm{max}}$, respectively. Assume that there is a $G_j-v_j$ such that $c_{\mathrm{max}}-c_j=1$. Since $|V(G_1)|=\cdots=|V(G_k)|$ and there are only $\ell-k$ isolated vertices in the disjoint union of $G_1-v_i,\dots,G_k-v_k$, there must be at least one component of order more than $1$ in each $G_i$. Thus, $p(G)\geq k$. Furthermore, choosing $G_j=N_{c_j+1}+K_{s-c_j-1}$ for $j=1,\dots,k$ where $s\geq c_{\mathrm{max}}+3$, we have $|V(G_1)|=\cdots=|V(G_k)|=s$ and so $p(G)=k$. On the other hand, suppose that $c_{\mathrm{max}}-c_j\neq1$ for all $1\leq j\leq k$. Choosing
\begin{align*}
G_j=\begin{cases}
N_{c_j+1}+K_{c_\mathrm{max}-c_j}, & \text{if $c_{\mathrm{max}}-c_j\geq 2$},\\
N_{c_\mathrm{max}+1}, &\text{if $c_j=c_{\mathrm{max}}$},
\end{cases}
\end{align*}
for $1\leq j\leq k$, we obtain $|V(G_1)|=\cdots=|V(G_k)|\geq 2$ and so $p(G)=k-q(G)$ where $q(G)\geq 1$.

Let $\cG_{\ell,k}$ be the set of graphs $G\in\cA_{\ell,k}$ such that for the elementary graphs $G_1,\dots,G_k$, $c_{\mathrm{max}}-c_j\neq1$ for all $1\leq j\leq k$, where $\ell-k\geq 1$. Then, we immediately have the following proposition.
\begin{proposition}\label{cG}
	Suppose that $G\in\cA_{\ell,k}$ where $\ell-k\geq 1$. If $G\in\cG_{\ell,k}$, then $p(G)\geq k-q(G)$ where $q(G)\geq 1$, and there exists a graph $H\in\cG_{\ell,k}$ such that $p(H)= k-q(G)$ where $q(G)\geq 1$. If $G\notin\cG_{\ell,k}$, then $p(G)\geq k$ and there exists a graph $H\in\cA_{\ell,k}$ such that $p(H)=k$.
\end{proposition}

Proposition \ref{cG} implies that if $\cG_{\ell,k}$ is non--empty, then $p_{\ell,k}<k$. Otherwise, $p_{\ell,k}=k$, and so $m_{\ell,k}=\ell$.

\begin{lemma}[Case (ii)]\label{lemma2:glb}
	Let $G\in\cA_{\ell,k}$. If $k=\ell$ or $k=\ell-1\geq 2$, then $m_{\ell,k}=\ell$.
\end{lemma}
\begin{proof}
Let $G_1,\dots,G_k$ be the elementary graphs of $G$. Suppose that $k=\ell$. Note that $|V(G_i)|\geq 2$ for $i=1,\dots,k$. Since each $G_i$ for $i=1,\dots,k$ has exactly one isolated vertex, every $G_i$ must have at least one component of order more than $1$. Thus, $p_{\ell,\ell}=k$, and so $m_{\ell\ell}=\ell$. If $k=\ell-1\geq 2$, there exists a graph $G_j$ for some $1\leq j\leq k$ such that $c_{\mathrm{max}}-c_j=1$. So, $\cG_{\ell,k}$ is the empty set, which implies that $m_{\ell,\ell-1}=\ell$.
\end{proof}

\begin{example}\label{eg:max q} {\rm{
	Let $G\in\cA_{16,5}$, and let $G_1,\dots,G_5$ be the elementary graphs of $G$. Note that each $G_i$ for $i=1,\dots,5$ has at least one isolated vertex. See the following configurations of two distributions of the $16$ vertices into $G_1,\dots,G_5$:
	
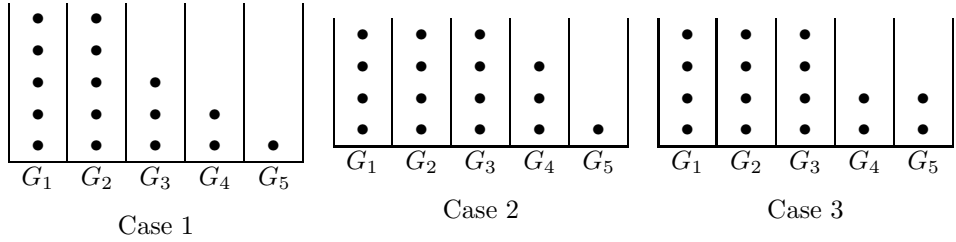
\[
	\begin{blockarray}{ccccc}
	\begin{block}{|c|c|c|c|c|}
	\bullet&\bullet&&&\\
	\bullet&\bullet&&&\\
	\bullet & \bullet&\bullet&&\\
	\bullet& \bullet&\bullet&\bullet& \\
	\bullet& \bullet&\bullet&\bullet& \bullet\\\hline
	\end{block}
	G_1 & G_2 & G_3 & G_4 & G_5 \\
	\noalign{\vspace{1.5ex}}
	\BAmulticolumn{5}{c}{%
		{\text{Case 1}}
	}
	\end{blockarray}
	\;\;\;\;
	\begin{blockarray}{ccccc}
	\begin{block}{|c|c|c|c|c|}
	\bullet&\bullet&\bullet&&\\
	\bullet & \bullet&\bullet&\bullet&\\
	\bullet& \bullet&\bullet&\bullet& \\
	\bullet& \bullet&\bullet&\bullet& \bullet\\\hline
	\end{block}
	G_1 & G_2 & G_3 & G_4 & G_5 \\
	\noalign{\vspace{1.5ex}}
	\BAmulticolumn{5}{c}{%
		{\text{Case 2}}
	}
	\end{blockarray}
	\;\;\;\;
	\begin{blockarray}{ccccc}
	\begin{block}{|c|c|c|c|c|}
	\bullet&\bullet&\bullet&&\\
	\bullet & \bullet&\bullet&&\\
	\bullet& \bullet&\bullet&\bullet&\bullet \\
	\bullet& \bullet&\bullet&\bullet& \bullet\\\hline
	\end{block}
	G_1 & G_2 & G_3 & G_4 & G_5 \\
	\noalign{\vspace{1.5ex}}
	\BAmulticolumn{5}{c}{%
		{\text{Case 3}}
	}
\end{blockarray}
	\]
	where for each case, a $\bullet$ indicates an isolated vertex and the $j^\text{th}$ column describes how many isolated vertices $G_j$ has. For Case 1, $G\in \cG_{16,5}$ and by Proposition \ref{cG}, we may have $p(G)=3$. Suppose that $G$ corresponds to the configuration of Case 2. Since $c_{\mathrm{max}}-c_4=1$, $G\notin\cG_{\ell,k}$ and so $p(G)\geq 5$. If $G$ corresponds to Case 3, then $c_{\mathrm{max}}-c_j\neq1$ for all $1\leq j\leq 5$ so that we can obtain $p(G)=2$ by placing $K_2$ in $G_4$ and $G_5$, respectively. Furthermore, there is no graph in $G\in\cG_{16,5}$ such that $c_\mathrm{max}=2$, by the pigeonhole principle. Therefore, $p_{16,5}=2$ and so $m_{16,5}=13$.
	
	Let $H\in\cA_{15,4}$, and let $H_1,\dots,H_4$ be the elementary graphs of $H$. Consider the following configurations of two distributions of the $15$ vertices into $H_1,\dots,H_4$:
	
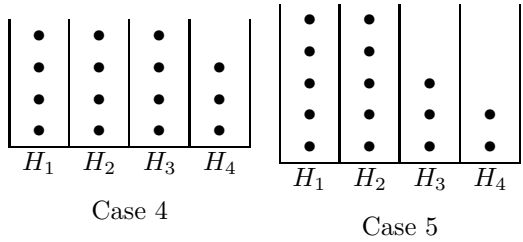
\[
	\begin{blockarray}{cccc}
	\begin{block}{|c|c|c|c|}
	\bullet&\bullet&\bullet&\\
	\bullet & \bullet&\bullet&\bullet\\
	\bullet& \bullet&\bullet&\bullet \\
	\bullet& \bullet&\bullet&\bullet\\\hline
	\end{block}
	H_1 & H_2 & H_3 & H_4 \\
	\noalign{\vspace{1.5ex}}
	\BAmulticolumn{4}{c}{%
		{\text{Case 4}}
	}
	\end{blockarray}
	\;\;\;\;
	\begin{blockarray}{cccc}
	\begin{block}{|c|c|c|c|}
	\bullet&\bullet& &\\
	\bullet&\bullet&&\\
	\bullet & \bullet&\bullet&\\
	\bullet& \bullet&\bullet&\bullet \\
	\bullet& \bullet&\bullet&\bullet\\\hline
	\end{block}
	H_1 & H_2 & H_3 & H_4 \\
	\noalign{\vspace{1.5ex}}
	\BAmulticolumn{4}{c}{%
		{\text{Case 5}}
	}
	\end{blockarray}
	\]
	For Case 4, $H\notin \cG_{15,4}$, so $p(H)\geq 4$. For Case 5, we have $p(H)\geq 2$. One can check that $m_{15,4}=13$.}}
\end{example}

Observe from Cases 1, 2 and 3 in Example \ref{eg:max q} that $c_{\mathrm{max}}(G)$ should be minimized in order to maximize $q(G)$ so that $p_{\ell,k}$ can be attained. So, we shall consider graphs $G\in\cA_{\ell,k}$ such that $0\leq\ell-k-c_{\mathrm{max}}(G)q(G)\leq c_{\mathrm{max}}(G)-1$, and then investigate the minimum of $c_{\mathrm{max}}(G)$ among the graphs $G$. However, Cases 4 and 5 in Example \ref{eg:max q} show that the minimum of $c_{\mathrm{max}}(G)$ being attained at $\hat{G}$ does not guarantee  attaining $p_{\ell,k}$ if $\ell-k=c_{\mathrm{max}}(\hat{G})q(\hat{G})-1$.

\begin{lemma}[Case (iii)]\label{lemma3:glb}
	Let $G\in\cA_{\ell,k}$ where $k\nmid \ell$ and $2\leq k\leq \ell-2$. Let $\tilde{c}=\mathrm{max}\{\left\lceil\frac{\ell-k}{k}\right\rceil,2\}$. Then, 
	\begin{equation*}
	m_{\ell,k}=\begin{cases}
	\ell-\lfloor\frac{\ell-k}{3}\rfloor, & \text{if $\ell-k$ is odd, and $\lfloor\frac{\ell-k}{2}\rfloor\leq k-1$},\\
		\ell-\lfloor\frac{k(\ell-k)}{\ell+k+1}\rfloor, & \text{if $k\mid (\ell+1)$, and $\ell+1\geq 4k$},\\
	\ell-\lfloor\frac{\ell-k}{\tilde{c}}\rfloor, & \text{otherwise}.
	\end{cases}
	\end{equation*}
\end{lemma}
\begin{proof}
	Let us consider a graph $G\in\cA_{\ell,k}$. Then, there exist the elementary graphs $G_1,\dots,G_k$ of $G$. Suppose that $0\leq\ell-k-c_{\mathrm{max}}(G)q(G)\leq c_{\mathrm{max}}(G)-1$ where $k\nmid \ell$ and $2\leq k\leq \ell-2$. We may assume that $c_1=\cdots=c_{q(G)}=c_\mathrm{max}(G)$ and $c_{q(G)+1}=r(G)$ where $r(G)=\ell-k-c_{\mathrm{max}}(G)q(G)$. Note that if $0\leq r(G)\leq c_{\mathrm{max}}(G)-2$, then $G\in\cG_{\ell,k}$. 
	
	Let $c_0=\mathrm{min}\{c\geq2|\lfloor\frac{\ell-k}{c}\rfloor\leq k-1\}$ and $r_0=\ell-k-c_0\lfloor\frac{\ell-k}{c_0}\rfloor$. We shall consider $3$ cases: (a) $c_0=2$ and $r_0=1$, (b) $\lfloor\frac{\ell-k}{c_0}\rfloor=k-1$ and $r_0=c_0-1$ where $c_0\geq 3$, (c) neither (a) nor (b) holds.
	\begin{itemize}
		\item (case (a)) If $c_\mathrm{max}(G)=2$ and $r(G)=1$, then $c_{\mathrm{max}}(G)-c_{q(G)+1}=1$ so that $p(G)\geq k$. Suppose that $c_{\mathrm{max}}(G)=3$. Since $c_0=2$ and $r_0=1$, $\lfloor\frac{\ell-k}{2}\rfloor\leq k-1$ implies that $\lfloor\frac{\ell-k}{3}\rfloor\leq k-2$. If $r(G)=0$ or $r(G)=1$, then $G\in\cG_{\ell,k}$ and by Proposition \ref{cG}, $p_{\ell,k}=k-\lfloor\frac{\ell-k}{3}\rfloor$. Assume that $r(G)=2$. Since $\lfloor\frac{\ell-k}{3}\rfloor\leq k-2$, there exists a graph $\hat{G}\in\cG_{\ell,k}$ such that $c_1(\hat{G})=\cdots=c_{q(G)}(\hat{G})=3$ and $c_{k-1}(\hat{G})=c_k(\hat{G})=1$. By Proposition \ref{cG}, we find that $m_{\ell,k}=\ell-\lfloor\frac{\ell-k}{3}\rfloor$. Furthermore, considering $c_0=2$, the condition $r_0=1$ is equivalent for $\ell-k$ to be odd.
		
		\item (case (b)) If $c_\mathrm{max}(G)=c_0\geq 3$, $q(G)=k-1$ and $r(G)=c_0-1$, then $c_\mathrm{max}-c_k=1$ so that $G\notin\cG_{\ell,k}$. Note that $\ell-k=c_0(k-1)+c_0-1$ can be expressed as $c_0=\frac{\ell+1}{k}-1\geq 3$, \textit{i.e.}, $\ell+1$ is divisible by $k$ and $\ell+1\geq 4k$. Suppose that $c_\mathrm{max}(G)=c_0+1$. We have $q(G)=\lfloor\frac{\ell-k}{c_0+1}\rfloor=\lfloor\frac{k(\ell-k)}{\ell+k+1}\rfloor$. Since $\lfloor\frac{\ell-k}{c_0}\rfloor=k-1$, we have $q(G)\leq k-2$. If $r(G)=0$, there exists $\hat{G}\in\cG_{\ell,k}$ such that $c_1(\hat{G})=\cdots=c_{q(G)}(\hat{G})=c_0+1$. If $r(G)\geq 1$, choose a graph $\hat{G}\in\cG_{\ell,k}$ such that $c_1(\hat{G})=\cdots=c_{q(G)}(\hat{G})=c_0+1$, $c_{k-1}(\hat{G})=r(G)-1$ and $c_{k}(\hat{G})=1$. Hence, by Proposition \ref{cG}, $m_{\ell,k}=\ell-\lfloor\frac{k(\ell-k)}{\ell+k+1}\rfloor$.
		
		\item (case (c)) Considering the cases (a) and (b), if $c_0=2$, then $r_0=0$; if $r_0=c_0-1$, then $\lfloor\frac{\ell-k}{c_0}\rfloor\leq k-2$. Let $c_\mathrm{max}(G)=c_0$ and $q(G)=\lfloor\frac{\ell-k}{c_0}\rfloor$. It is readily checked that for $c_0=2$ we can obtain our desired result. If $r(G)=c_0-1\geq 2$, then $q(G)\leq k-2$. Then, there exists a graph $\hat{G}\in\cG_{\ell,k}$ such that $c_1(\hat{G})=\cdots=c_{q(G)}(\hat{G})=c_0$, $c_{k-1}(\hat{G})=r(G)-1$ and $c_{k}(\hat{G})=1$. If $r(G)<c_0-1$, it is straightforward that $G\in\cG_{\ell,k}$. Therefore, $m_{\ell,k}=\ell-\lfloor\frac{\ell-k}{c_0}\rfloor$. Consider $c_0=\mathrm{min}\{c\geq2|\lfloor\frac{\ell-k}{c}\rfloor\leq k-1\}$. Since $\left\lfloor\frac{\ell-k}{c}\right\rfloor\leq k-1\Leftrightarrow\frac{\ell-k}{c}<k\Leftrightarrow \frac{\ell-k}{k}<c$, we have $c_0=\mathrm{max}\{\left\lceil\frac{\ell-k}{k}\right\rceil,2\}$.
	\end{itemize} 
\end{proof}

Summarizing Lemmas \ref{lemma1:glb}, \ref{lemma2:glb} and \ref{lemma3:glb}, we have the following theorem.

\begin{theorem}\label{greatestlowerbounds}
	Let $G\in\cA_{\ell,k}$ where $1\leq k\leq \ell$. Then,
	\begin{numcases}{m_{\ell,k}=}\label{eq1:glb}
	\ell, & \text{if $k=\ell-1\geq 2$ or $k=\ell$},\\\label{eq2:glb}
	\ell-k, & \text{if $k\mid \ell$ and $1\leq k< \ell$},\\\label{eq3:glb}
	\ell-\left\lfloor\frac{k(\ell-k)}{\ell+k+1}\right\rfloor, & \text{if $k\mid (\ell+1)$, $\ell+1\geq 4k$, $2\leq k\leq \ell-2$},\\\label{eq4:glb}
	\ell-\left\lfloor\frac{\ell-k}{3}\right\rfloor, & \text{if $k\nmid \ell$, $2\nmid (\ell-k)$, $\lfloor\frac{\ell-k}{2}\rfloor\leq k-1\leq\ell-3$},\\\label{eq5:glb}
	\ell-\left\lfloor\frac{\ell-k}{\tilde{c}}\right\rfloor, & \text{otherwise},
	\end{numcases}
	where $\tilde{c}=\mathrm{max}\{\left\lceil\frac{\ell-k}{k}\right\rceil,2\}$.
\end{theorem}

\begin{corollary}\label{glb:am&l}
	Let $G$ be a non--complete connected graph of order $n$ with $i(G)=1$ and $\ell\geq 1$ vertices of $\delta(G)$. Then,
	\begin{equation*}
	am(\alpha(G))\geq\begin{cases}
	\frac{\ell}{2}, & \text{$\ell$ is even},\\
	\ell-\lfloor\frac{\ell}{3}\rfloor, & \text{$\ell$ is odd}.
	\end{cases}
	\end{equation*}
	with equality for even $\ell$ if and only if $G=\vee_{i=1}^{\frac{\ell}{2}}N_2$ ($\ell\geq 4$) or $G=(\vee_{i=1}^{\frac{\ell}{2}}N_2)\vee K_{n-\ell}$. In particular, $G=N_2\vee K_{n-2}$ for $\ell=2$.
\end{corollary}
\begin{proof}
	Let $m_\ell:=\mathrm{min}\{am(\alpha(G))|G\in\cA_\ell\}$. We need only  find $m_\ell$ for even $\ell$ and odd $\ell$, respectively, to complete the proof. Continuing the notation  of Theorem \ref{greatestlowerbounds}, for the case \eqref{eq3:glb}, there exists $a\geq 1$ such that $\ell+1=ak$. Since $k\leq\ell$, we have $a\geq 2$. So, $\ell-\lfloor\frac{k(\ell-k)}{\ell+k+1}\rfloor$ can be recast as $\ell-\lfloor\frac{(\ell-k)}{a+1}\rfloor\geq \ell-\lfloor\frac{(\ell-k)}{3}\rfloor$, \textit{i.e.}, $\lfloor\frac{(\ell-k)}{a+1}\rfloor\leq \lfloor\frac{(\ell-k)}{3}\rfloor$.
	
	Suppose that $\ell$ is even. Then, $\frac{\ell}{2}\mid\ell$. From \eqref{eq2:glb}, we have $m_{\ell,\frac{\ell}{2}}=\ell-\frac{\ell}{2}$ with $k=\frac{\ell}{2}$. Note that $\tilde{c}\geq 2$. So, we have $\lfloor\frac{(\ell-k)}{3}\rfloor<\frac{\ell}{2}$ and $\lfloor\frac{(\ell-k)}{\tilde{c}}\rfloor<\frac{\ell}{2}$ for $1\leq k\leq \ell$. Hence, $m_\ell=\ell-\frac{\ell}{2}$, which is only attained from \eqref{eq2:glb}. Furthermore, we find from Lemma \ref{lemma1:glb} that $am(\alpha(G))=\frac{\ell}{2}$ for $G\in\cA_{\ell}$ if and only if $G=\vee_{i=1}^{\frac{\ell}{2}}N_2$ ($\ell\geq 4$) or $G=(\vee_{i=1}^{\frac{\ell}{2}}N_2)\vee G'$ where $\alpha(G')>|V(G')|-2$. It follows from $\delta(G')\leq |V(G')|-1$ that $G'$ is the complete graph.
	
	It is straightforward that $m_1=1$. Assume that $\ell$ is odd and $3\mid \ell$. Applying \eqref{eq2:glb}, $m_{\ell,\frac{\ell}{3}}=\ell-\frac{\ell}{3}$. Suppose that for \eqref{eq5:glb}, there are $\tilde{c}\geq 2$ and $k_0\geq 1$ such that $\ell\neq 3k_0$ and $\left\lfloor\frac{\ell-k_0}{\tilde{c}}\right\rfloor\geq\frac{\ell}{3}$. Since $k_0\geq 1$, we must have $\tilde{c}=2$. This implies that $\ell>3k_0$. So, $\left\lceil\frac{\ell-k_0}{k_0}\right\rceil>2$, which is a contradiction to $\tilde{c}=\mathrm{max}\{\left\lceil\frac{\ell-k_0}{k_0}\right\rceil,2\}=2$. Hence, $\left\lfloor\frac{\ell-k}{\tilde{c}}\right\rfloor< \frac{\ell}{3}$. Furthermore, since $\lfloor\frac{(\ell-k)}{3}\rfloor<\frac{\ell}{3}$ for $1\leq k\leq \ell$, we have $m_\ell=\ell-\frac{\ell}{3}$.
	
	Suppose that $\ell$ is odd and $\ell=3b+1$ for some $b\geq 2$. Choose $k=b+1$ so that $\ell-k=2b$. Thus, $\lfloor\frac{\ell-k}{2}\rfloor=b$ and by \eqref{eq5:glb}, $m_{\ell,b+1}=\ell-\lfloor\frac{\ell}{3}\rfloor$. If $k$ is in the case of \eqref{eq2:glb}, then $k$ ($\neq \ell$) is a divisor of $\ell$. Then, $k=1$ or $k\geq 5$. Note that $\ell$ is odd and $\ell\geq 7$. It follows that $k<\left\lfloor\frac{\ell}{3}\right\rfloor$ for all divisors $k$ ($\neq\ell$) of $\ell$. Moreover, since we have $\lfloor\frac{(\ell-k)}{3}\rfloor<\lfloor\frac{\ell}{3}\rfloor$ for $k\geq 2$, $m_{\ell,b+1}<m_{\ell,k}$ for any $k$ corresponding to \eqref{eq3:glb} or \eqref{eq4:glb}. Therefore, $m_\ell=\ell-\lfloor\frac{\ell}{3}\rfloor$. 
	
	Similarly, assume that $\ell$ is odd and $\ell=3d+2$ for some $d\geq 1$. Choose $k=d+2$. Then, $\ell-k=2d$ which implies from \eqref{eq5:glb} that $m_{\ell,d+2}=\ell-\lfloor\frac{\ell}{3}\rfloor$. Note that $\ell\geq 5$. For \eqref{eq2:glb}, let $k$ ($\neq \ell$) be a divisor of $\ell$. Then, $k\leq\lfloor\frac{\ell}{3}\rfloor$ with equality if and only if $k=1$ and $\ell=5$. Furthermore, $\lfloor\frac{(\ell-k)}{3}\rfloor\leq\lfloor\frac{\ell}{3}\rfloor$ for $k\geq 2$ with equality if and only if $k=2$. In particular, one can verify that if $k=2$, then $k$ falls under \eqref{eq3:glb}, and $\left\lfloor\frac{k(\ell-k)}{\ell+k+1}\right\rfloor=\lfloor\frac{\ell}{3}\rfloor$ if and only if $\ell=5$. Hence, $m_{\ell,d+2}\leq m_{\ell,k}$ for any $k$ corresponding to \eqref{eq3:glb} or \eqref{eq4:glb} with equality if and only if $k=2$ and $\ell=5$. 
\end{proof}

\begin{remark}{\rm{
	Continuing the notation of Corollary \ref{glb:am&l}, graphs attaining the equality for odd $\ell$ can be classified by the proof in Corollary \ref{glb:am&l}. Suppose that $3\mid\ell$. By Lemma \ref{lemma1:glb}, $G=\vee_{i=1}^{\frac{\ell}{3}}N_3$ for $\ell\geq 6$ or $G=(\vee_{i=1}^{\frac{\ell}{3}}N_3)\vee G'$ where $\alpha(G')>|V(G')|-3$. Assume that $\ell$ is odd and $\ell=3b+1$ for some $b\geq 2$. Since $\ell \geq 7$, the equality is only attained by the case \eqref{eq5:glb}. Hence, $G=(\vee_{i=1}^{b}N_3)\vee(N_1+K_2)$ or $G=(\vee_{i=1}^{b}N_3)\vee(N_1+K_2)\vee G'$ where $\alpha(G')>|V(G')|-3$. Suppose that $\ell=3d+2$ for some $d\geq 1$. For $\ell=5$, the equality holds with $k=1,2,3$. Thus, we have following cases: for $k=1$, $G=N_5\vee G'$ where $\alpha(G')>|V(G')|-5$; for $k=2$, $G=N_4\vee(N_1+ K_3)$, $G=N_4\vee(N_1+(N_1\vee K_2))$, $G=N_4\vee(N_1+K_3)\vee G'$ or $G=N_4\vee(N_1+(N_1\vee K_2))\vee G'$ where $\alpha(G')>|V(G')|-4$; for $k=3$, $G=N_3\vee(N_1+K_2)\vee(N_1+K_2)$ or $G=N_3\vee(N_1+K_2)\vee(N_1+K_2)\vee G'$ where $\alpha(G')>|V(G')|-3$. For $\ell\geq 11$, it can be checked that $m_\ell$ is only attained by $G=(\vee_{i=1}^dN_3)\vee(N_1+K_2)\vee(N_1+K_2)$ or $G=(\vee_{i=1}^dN_3)\vee(N_1+K_2)\vee(N_1+K_2)\vee G'$ where $\alpha(G')>|V(G')|-3$.}}
\end{remark}

The following theorem is our main result in this section for classifying graphs $G$ with $i(G)=1$ and $am(\alpha(G))=1$.

\begin{theorem}\label{extreme:i(G)=am(a(G))=1}
	Let $G$ be a non--complete connected graph of order $n$. Then, $i(G)=1$ and $am(\alpha(G))=1$ if and only if either $G=N_2\vee K_{n-2}$  or $G=G_1\vee G'$ where $G_1$ is a graph of order $n-\delta(G)$ with exactly one isolated vertex, and $G'$ is a graph on $\delta(G)$ vertices with $\alpha(G')>2\delta(G)-n$ and $\delta(G')>2\delta(G)-n$.
\end{theorem}
\begin{proof}
	Suppose that $i(G)=1$ and $am(\alpha(G))=1$. Let $\ell$ be the number of vertices of the minimum degree in $G$. By Corollary \ref{glb:am&l}, $\ell=1$ or $\ell=2$. For $\ell=1$, since $G$ is connected, $G$ is a $1$--join with $G'$. Since $\mathrm{deg}_G(v)>\delta(G)$ for all $v\in V(G')$, we have $\delta(G')>2\delta(G)-n$. The hypothesis that $am(\alpha(G))=1$ implies that $\alpha(G')>2\delta(G)-n$. For $\ell=2$, the conclusion is clear from Corollary \ref{glb:am&l}.
	
	It is straightforward to prove the converse.
\end{proof}

\begin{example}\label{example:i(G)=am(a(G))=1}{\rm{
	Suppose that $G_1=K_{n_1}+N_1$ and $G'=K_{n_2}$ where $n_1,n_2>0$. Consider $G=G_1\vee G'$. Then, $\alpha(G')=n_2$, $\delta(G')=n_2-1$ and $2\delta(G)-|V(G)|=n_2-n_1-1$. By Theorem \ref{extreme:i(G)=am(a(G))=1}, we have $i(G)=1$ and $am(\alpha(G))=1$.}}
\end{example}

Now, we shall introduce a result without proof, as well as some notation in \cite{Urschel:bisection}, to find pathological graphs with respect to applying spectral bisection for the graph partitioning problem. Let $G$ be a connected graph of order $n$, and let $X$ be the eigenspace corresponding to $\alpha(G)$, and denote
\begin{align*}
i_+(\bx)&:=\{j|1\leq j\leq n, x_j>0\},\\
i_-(\bx)&:=\{j|1\leq j\leq n, x_j<0\},\\
i_0(\bx)&:=\{j|1\leq j\leq n, x_j=0\},\\
i_0(X)&:=\bigcap_{\bx\in X}i_0(\bx).
\end{align*}



\begin{theorem}[\cite{Urschel:bisection}]\label{urschel:theroem}
	Let $G$ be a connected graph. Then, there exists a Fiedler vector $\bx$ such that the subgraphs of $G$ induced by $i_+(\bx)\cup i_0(\bx)$ and $i_-(\bx)$ are connected.
\end{theorem}

\begin{proposition}\label{join:zero component}
	Let $G$ be a connected graph of order $n$, and $X$ be the eigenspace corresponding to $\alpha(G)$. Suppose that there exists an induced subgraph $G_2$ of $G$ such that $G=G_1\vee G_2$ and $\alpha(G_2)>\alpha(G)-|V(G_1)|$. Then, $V(G_2)\subseteq i_0(X)$.
\end{proposition}
\begin{proof}
	Considering eigenvectors of the join of graphs and the condition that $\alpha(G_2)>\alpha(G)-|V(G_1)|$, it implies that for any Fiedler vector, vertices of $V(G_2)$ are valuated by $0$. Hence, $V(G_2)\subseteq i_0(X)$.
\end{proof}

\begin{example}\label{remark:almost}{\rm{
	The converse of Proposition \ref{join:zero component} does not hold for the following graph $G$:
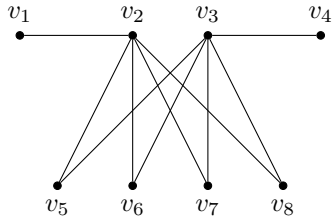
	\begin{center}
		\begin{tikzpicture}
		\tikzset{enclosed/.style={draw, circle, inner sep=0pt, minimum size=.10cm, fill=black}}
		
		\node[enclosed, label={above, yshift=0cm: $v_1$}] (v_1) at (0,2) {};
		\node[enclosed, label={above, yshift=0cm: $v_2$}] (v_2) at (1.5,2) {};
		\node[enclosed, label={above, yshift=0cm: $v_3$}] (v_3) at (2.5,2) {};
		\node[enclosed, label={above, yshift=0cm: $v_4$}] (v_4) at (4,2) {};
		\node[enclosed, label={below, yshift=0cm: $v_5$}] (v_5) at (0.5,0) {};
		\node[enclosed, label={below, yshift=0cm: $v_6$}] (v_6) at (1.5,0) {};
		\node[enclosed, label={below, yshift=0cm: $v_7$}] (v_7) at (2.5,0) {};
		\node[enclosed, label={below, yshift=0cm: $v_8$}] (v_8) at (3.5,0) {};

		\draw (v_1) -- (v_2);
		\draw (v_3) -- (v_4);
		\draw (v_2) -- (v_5);
		\draw (v_2) -- (v_6);
		\draw (v_2) -- (v_7);
		\draw (v_2) -- (v_8);
		\draw (v_3) -- (v_5);
		\draw (v_3) -- (v_6);
		\draw (v_3) -- (v_7);
		\draw (v_3) -- (v_8);
		\end{tikzpicture}.
	\end{center}
	Let $X$ be the eigenspace corresponding to $\alpha(G)$. It follows from computations that $\lambda_1(G)<|V(G)|=8$, $am(\alpha(G))=1$ and $i_0(X)=\{v_5,v_6,v_7,v_8\}$. Since $\lambda_1(G)<8$, $G$ cannot be expressed as a join. }}
\end{example}

Theorem \ref{urschel:theroem} provides the existence of a Fiedler vector preserving connectedness of the two subgraphs for any connected graph. However, this does not guarantee that such a Fiedler vector gives a partition into two subgraphs such that they are similar in size. Next, we will show a family of graphs such that sign patterns of all Fiedler vectors are extremely unbalanced. In Theorem \ref{urschel:theroem}, we may slightly change the condition for the result as follows: the subgraphs of $G$ induced by $i_-(\bx)\cup i_0(\bx)$ and $i_+(\bx)$ are connected.

\begin{example}\label{example:worse}{\rm{
	Suppose that $G$ is a non--complete connected graph of order $n$ with $i(G)=1$ and $am(G)=1$. Then, by Theorem \ref{extreme:i(G)=am(a(G))=1}, $G=N_2\vee K_{n-2}$ or $G$ is a $1$--join with $G'$ where $\alpha(G')>2\delta(G)-n$ and $\delta(G')>2\delta(G)-n$. For a Fiedler vector $\bx$ of $G=N_2\vee K_{n-2}$, without loss of generality, two subgraphs of $G$ induced by $i_-(\bx)\cup i_0(\bx)$ and $i_+(\bx)$ are $K_{n-1}$ and $N_1$, respectively. 
	
	For the latter case $G=G_1\vee G'$, let us revisit Example \ref{example:i(G)=am(a(G))=1}. Suppose that $X$ is the eigenspace corresponding to $\alpha(G)$ where $G=(K_{n_1}+N_1)\vee K_{n_2}$. By Proposition \ref{join:zero component}, we have $K_{n_2}\subseteq i_0(X)$. Since $am(\alpha(G))=1$, $i_0(X)=K_{n_2}$. From Theorem \ref{urschel:theroem}, we may have that $i_-(\bx)\cup i_0(\bx)$ and $i_+(\bx)$ are $K_{n_2+1}$ and $K_{n_1}$, respectively. Therefore, for pairs $(n_1,n_2)$ such that $\frac{n_1}{n_2}\rightarrow\infty$, the corresponding graph $G$ will be pathological with respect to spectral bisection.}}
\end{example}


\section{Some classes of graphs with $i(G)=1$}


In this section, we will consider threshold graphs and graphs with three distinct Laplacian eigenvalues in the context of $i(G)=1$.

\begin{definition}\label{def:thresh}
	A threshold graph is a graph obtained from a single vertex by repeatedly performing one of the following operations:
	\begin{enumerate}
		\item addition of a single isolated vertex to the graph
		\item addition of a dominating vertex.
	\end{enumerate} 
\end{definition}

\begin{proposition}
	Every connected threshold graph $G$ of order $n$ has $i(G)=1$.
\end{proposition}
\begin{proof}
	We will use induction on the number of vertices to complete the proof. If $G$ is a complete graph, we are done. Let $G$ be a non--complete connected threshold graph of order $n$. For order $3$, $N_2\vee N_1$ is the only such graph, and $i(N_2\vee N_1)=1$. Let $n>3$. Suppose that any non--complete connected threshold graph $H$ of order $k<n$ satisfies $i(H)=1$. Since $G$ is a connected threshold graph, there exists a vertex $v$ with $\mathrm{deg}(v)=n-1$. Let $G'=G-\{v\}$. Suppose that $G'$ is connected. Then, $G'$ is not complete, otherwise, $G$ would be complete. By induction, $i(G')=1$, and so $\delta(G')=\alpha(G')$. Considering the spectrum of $G'\vee\{v\}$, we have  $$\alpha(G)=\alpha(G')+1=\delta(G')+1=\delta(G).$$ Therefore, $i(G)=1$. If $G'$ is disconnected, then $G'$ has an isolated vertex. By Theorem \ref{equiv:i(G)=1}, $i(G)=1$.
\end{proof}

The spectrum of a threshold graph appears in \cite{Merris:degree maximal graph}. In the paper \cite{Merris:degree maximal graph}, a connected threshold graph is called a maximal graph since it is proved there that the degree sequence of a connected threshold graph of size $m$ is not majorized by any other degree sequences of graphs of size $m$. In particular, we will introduce the following results used for seeing how $am(\alpha(G))$ plays a role. 

\begin{theorem}[\cite{Merris:degree maximal graph}]\label{maximal:conjugate,laplacian}
	If $G$ is a connected threshold graph, then $S(G)=\mathbf{d}^*$ where $\mathbf{d}^*$ is the conjugate of the degree sequence of $G$.
\end{theorem}

\begin{theorem}[\cite{Merris:degree maximal graph}]\label{maximal:component,isolated}
	Let $G$ be a threshold graph. Suppose that $G$ is disconnected so that there are $\ell+1$ connected components. Then, $\ell$ components consist of isolated vertices.
\end{theorem}

\begin{proposition}
	Suppose that $G$ is a non--complete connected threshold graph of order $n$. Then, $\alpha(G)=k$ and $am(\alpha(G))=\ell$ if and only if there are exactly $k$ vertices $v_1.\dots,v_k$ so that $\mathrm{deg}_G(v_i)=n-1$ for $i=1,\dots,k$ and the subgraph $G_1$ of $G$ induced by $V(G)-\{v_1,\dots,v_k\}$ consists of $\ell+1$ components, $\ell$ components of which consist of a single vertex, respectively.
\end{proposition}
\begin{proof}
	Suppose that $\alpha(G)=k$ and $am(\alpha(G))=\ell$. By Theorem \ref{maximal:conjugate,laplacian}, the number of vertices of degree $n-1$ is $\alpha(G)$. There are exactly $k$ vertices $v_1, \dots,v_k$ such that $\mathrm{deg}_G(v_i)=n-1$ for $i=1,\dots,k$. Suppose that $G_1$ is the subgraph of $G$ induced by $V(G)-\{v_1,\dots,v_k\}$. Since there are only $k$ vertices of degree $n-1$ in $G$, the graph $G_1$ is disconnected. Moreover, $G=G_1\vee K_k$. Since $am(\alpha(G))=\ell$, from Theorem \ref{maximal:component,isolated}, we obtain the desired result.
	
	For the converse, evidently we have $G=G_1\vee K_k$. Since $G_1$ has exactly $\ell$ isolated vertices, $\alpha(G)=k$ and $am(\alpha(G))=\ell$.
\end{proof}

Now, we will investigate an equivalent condition for a graph $G$  that is a join having  three distinct Laplacian eigenvalues to have $i(G)=1$.

\begin{proposition}\label{three:distinct}
	Let $G$ be a non--complete, connected graph of order $n$. The graph $G$ has three  distinct Laplacian eigenvalues $0,\alpha(G)$ and $n$ where $am(\alpha(G))=k$ if and only if there exist integers $p\geq 0$, $q\geq1$ and $r\geq 2$ such that $p+q\geq 2$ and $G=K_p\vee(\vee_{i=1}^q N_r)$ where $n=qr+p$, $\alpha(G)=r(q-1)+p$ and $k=q(r-1)$.
\end{proposition}
\begin{proof}
	Suppose that $G$ has $3$ distinct Laplacian eigenvalues $0,\alpha(G)$ and $n$. Then, the complement $\bar{G}$ of $G$ has $n-k$ connected components since $\bar{G}$ has $0$ as an eigenvalue with multiplicity $n-k$. Hence, there are graphs $G_1,\dots, G_{n-k}$ such that $G=G_1\vee\cdots\vee G_{n-k}$ where $n-k\geq 2$. Note that for $i=1,\dots,n-k$, $L(G_i)$ does not have $|V(G_i)|$ as an eigenvalue. If there is a $G_j$ with three distinct eigenvalues, then  from the spectrum of a join of graphs, we find that $G$ has more than three distinct eigenvalues, a contradiction. So, each $G_i$ has either one or two distinct eigenvalues. The only graphs with one eigenvalue are empty graphs, and the only graphs with two distinct eigenvalues are complete graphs. So, each $G_i$ is either $N_{r_i}$ or $K_{p_i}$ for some $r_i$ or $p_i$. Consider $N_{r_i}$ and $N_{r_j}$ for $r_i,r_j\geq 2$ and $r_i\neq r_j$. Then, $L(N_{r_i}\vee N_{r_j})$ has $4$ distinct eigenvalues $0,r_i,r_j$ and $r_i+r_j$. Hence, all empty graphs as factors in $G_1\vee\cdots\vee G_{n-k}$ must have the same order. Evidently, $K_{p_i}\vee K_{p_j}=K_{p_i+p_j}$ for $p_i,p_j\geq 1$. If $G_i$ is a complete graph, then $G_i=K_1$. Let $p$ be the number of isolated vertices in $\bar{G}$, let $q$ be the number of the complete graphs of order $r\geq 2$ in $\bar{G}$. If $q=0$, then $G$ is a complete graph. So, $q\geq 1$. If $p+q=1$, then $G$ is disconnected and so $p+q\geq 2$. Therefore, we have the desired graph $G$. Considering the spectrum of a join of graphs, the remaining conditions for $n$, $\alpha(G)$ and $k$ can be checked.
	
	By the spectrum of a join, the proof of the converse is straightforward.
\end{proof}

\begin{corollary}
	Let $G$ be a non--complete, connected graph of order $n$ with three distinct Laplacian eigenvalues. The largest Laplacian eigenvalue is $n$ if and only if $i(G)=1$.
\end{corollary}
\begin{proof}
	Suppose that the largest Laplacian eigenvalue is $n$. From Proposition \ref{three:distinct}, there exist $p\geq 0$, $q\geq1$ and $r\geq 2$ such that $p+q\geq 2$ and $G=K_p\vee(\vee_{i=1}^q N_r)$. Since $G=N_r\vee(K_p\vee(\vee_{i=1}^{q-1} N_r))$, we obtain $i(G)=1$ by Theorem \ref{equiv:i(G)=1}. Conversely, $i(G)=1$ implies that $G$ is a join of some graphs. So, the largest eigenvalue is $n$.
\end{proof}

\begin{corollary}
	Let $G$ be a non--complete, connected graph of order $n$ with three distinct Laplacian eigenvalues $0$, $\alpha(G)$ and $n$ where $k=am(\alpha(G))$. Then, the clique number of $G$ is  $$\omega(G)=n-k.$$
\end{corollary}
\begin{proof}
	It follows from Proposition \ref{three:distinct} that there exist $p\geq 0$, $q\geq1$ and $r\geq 2$ such that $p+q\geq 2$ and $G=K_p\vee(\vee_{i=1}^q N_r)$. So, $\omega(G)=p+q$. Since $n=qr+p$ and $k=qr-q$, we have $\omega(G)=n-k$.
\end{proof}


\section{Characterization of regular graphs with $i(G)=2$}


In this section, we shall consider $i(G)=2$. It turns out that $i(K_n)=1$. So, if $i(G)=2$, then $G$ is non--complete and connected.

\begin{proposition}\label{i(G)=2}
	Let $G$ be a connected graph of order $n$ with $i(G)=2$, and $\mathbf{x}$ be a Fiedler vector with $i(\mathbf{x})=2$. Then, two vertices valuated by negative numbers of $\mathbf{x}$ are adjacent and $0<\delta(G)-\alpha(G)\leq 1$. Moreover, one of the two vertices has degree $\delta(G)$.
\end{proposition}
\begin{proof}
	Since $i(G)=2$, there exists $\mathbf{x}=(x_1\dots,x_n)^T\in \mathbb{R}^n$ such that $x_1,x_2<0$, $x_j\geq 0$ for $j=3,\dots, n$ and $(L(G)-\alpha(G)I)\mathbf{x}=0$. We have
	\begin{align}
	(\ell_{11}-\alpha(G))x_1+\ell_{12}x_2+\ell_{13}x_3+\cdots+\ell_{1n}x_n=0,\label{temeq1}\\
	\ell_{21}x_1+(\ell_{22}-\alpha(G))x_2+\ell_{23}x_3+\cdots+\ell_{2n}x_n=0\label{temeq2}.
	\end{align}
	Since $i(G)>1$, it follows that
	\begin{align}\label{temeq4}
	\ell_{ii}-\alpha(G)\geq\delta(G)-\alpha(G)>0
	\end{align}
	for $i=1,\dots,n$. Assume that $\ell_{12}=\ell_{21}=0$. Thus, $(\ell_{11}-\alpha(G))x_1<0$ and $\sum_{j=3}^n\ell_{1j}x_j\leq 0$, which leads to having the left--hand side of \eqref{temeq1} negative. Therefore, $\ell_{12}=\ell_{21}=-1$.
	
	Adding \eqref{temeq1} and \eqref{temeq2}, we have
	\begin{align}
	(\ell_{11}-\alpha(G)-1)x_1+(\ell_{22}-\alpha(G)-1)x_2+\sum_{j=3}^n(\ell_{1j}+\ell_{2j})x_j=0.\label{temeq3}
	\end{align}
	Without loss of generality, suppose that $\ell_{11}\leq\ell_{22}$. If $\ell_{11}-\alpha(G)>1$, then the left--hand side of the equation \eqref{temeq3} is negative. Therefore, $\ell_{11}-\alpha(G)\leq 1$ and by \eqref{temeq4}, $0<\delta(G)-\alpha(G)\leq 1$. Furthermore, suppose that $\ell_{11}>\delta(G)$, that is, $\ell_{11}\geq \delta(G)+1$. Using  $\ell_{11}-\alpha(G)\leq 1$, we  deduce $\alpha(G)=\delta(G)$, which is a contradiction to $i(G)=2$. Thus, $\ell_{11}=\delta(G)$.
\end{proof}

\begin{remark}\label{remark:i(G)=2}{\rm{
	Proposition \ref{i(G)=2} provides two cases: $0<\delta(G)-\alpha(G)<1$ and $\delta(G)-\alpha(G)=1$. Note that $\delta(G)\geq v(G)\geq\alpha(G)$. Consider the case $0<\delta(G)-\alpha(G)<1$. Since $\alpha(G)$ is not an integer, we have $\delta(G)=v(G)>\alpha(G)$. 
	
	Suppose that $\delta(G)-\alpha(G)=1$. Then, continuing the notation and hypothesis in the proof of Proposition \ref{i(G)=2}, it follows from  \eqref{temeq3} that $\ell_{22}\leq \alpha(G)+1=\delta(G)$; by $\ell_{22}\geq\delta(G)$, we have $\ell_{22}=\delta(G)$. Hence, the two vertices valuated by negative signs of a Fiedler vector $\mathbf{x}$ in Proposition \ref{i(G)=2} have degree $\delta(G)$. Furthermore, we have either $\delta(G)-v(G)=0$ or $\delta(G)-v(G)=1$. For the latter case, since $\delta(G)-\alpha(G)=1$, we have $v(G)=\alpha(G)$. It follows from \cite{Steve:connectivity} that $G$ can be written as a join of two graphs $G_1$ and $G_2$ such that $G_1$ is a disconnected graph of order $n-v(G)$ and $G_2$ is a graph on $v(G)$ vertices with $\alpha(G_2)\geq 2v(G)-n$.}}
\end{remark}

Recall that $\lambda_k(G)$ and $\mu_k(G)$ are $k^{\text{th}}$--Laplacian and $k^{\text{th}}$--adjacency eigenvalues in the sequences of eigenvalues $S(L(G))$ and $S(A(G))$ in non--increasing order, respectively. We shall consider a connected $r$--regular graph $G$ of order $n$ with $i(G)=2$. Note that $L(G)=rI-A(G)$. So, $\alpha(G)=r-\mu_2(G)$ where $\mu_2<r$, and any Fiedler vector of $G$ is an eigenvector of $A(G)$ associated to $\mu_2$. Therefore, we also use eigenvectors associated to the second largest eigenvalue of $A(G)$ as Fiedler vectors without distinction.

A matching in a graph $G$ is a set of edges in $G$ such that no two edges in the set share a common vertex.

\begin{proposition}\label{indep2}
	Let $G$ be a connected $r$--regular graph $G$ of order $n$ with $i(G)=2$. Then,
	$$0<\mu_2(G)\leq 1.$$ In particular, if $\mu_2(G)=1$, then there is a matching of size at least $2$ in $G$.
\end{proposition}
\begin{proof}
	Consider $\alpha(G)=r-\mu_2(G)$ and $\delta(G)=r$. It is straightforward from Proposition \ref{i(G)=2} that $0<\mu_2(G)\leq 1$. 
	
	Suppose that $\mu_2(G)=1$. Since $i(G)=2$, there exists $\bx\in \mathbb{R}^n$ such that $(A(G)-\mu_2(G)I)\bx=\mathbf{0}$ and $i(\bx)=2$. We may assume that $\mathbf{x}=(x_1\dots,x_n)^T\in \mathbb{R}^n$ such that $x_1,x_2<0$, $x_j\geq 0$ for $j=3,\dots, n$. Let $A(G)=[a_{ij}]_{n\times n}$. By Proposition \ref{i(G)=2}, we have $a_{12}=a_{21}=1$. From the equations in the first and second rows of $(A(G)-\mu_2(G)I)\bx=\mathbf{0}$,
	\begin{equation*}
		-x_1+x_2+\sum_{j=3}^{n}a_{1j}x_j=0\;\text{and}\;\; x_1-x_{2}+\sum_{j=3}^{n}a_{2j}x_j=0. 
	\end{equation*}
	Adding the two equations, we obtain
	$$\sum_{j=3}^{n}a_{1j}x_j+\sum_{j=3}^{n}a_{2j}x_j=0.$$
	Since $x_j\geq 0$ for $j=3,\dots,n$ and $A(G)\geq 0$, it follows that $\sum_{j=3}^{n}a_{1j}x_j=\sum_{j=3}^{n}a_{2j}x_j=0$ and $x_k=0$ for any vertex $v_k$ adjacent to $v_1$ or $v_2$. Furthermore, $x_1=x_2$. Let $I=\{k\in [n]|x_k>0\}$ where $[n]=\{1,\dots,n\}$, and let $\tilde{A}$ be the corresponding principal submatrix $A[I]$ and $\tilde{\bx}$ be the corresponding subvector $\bx[I]$. Then, $\tilde{A}\tilde{\bx}=\tilde{\bx}$ where $\tilde{\bx}>0$. Suppose that a subgraph $H$ associated with $\tilde{A}$ is connected. By the Perron--Frobenius theorem, the eigenvalue $1$ is the spectral radius of $\tilde{A}$ and is simple. It implies that $H=K_2$. Since any vertex $v_k$ for $k\in I$ is not adjacent to $v_1$ and $v_2$, there are two edges, namely $v_1\sim v_2$ and the edge in $H$, such that they do not share any vertex in common. Next, assume that $H$ is disconnected. Since each component of $H$ is connected, $H$ consists of pairwise non--adjacent edges. Therefore, $G$ contains at least $2$ pairwise non--adjacent edges. 
\end{proof}

It can be found in \cite{Cvetkovic:second} that $\mu_2(K_{n_1,n_2,\dots,n_k})=0$, where $\max(n_1,n_2,\dots,n_k)\geq 2$, $\mu_2(K_n)=-1$, and $\mu_2(G)>0$ for all other connected graphs $G$. It is clear that $i(K_n)=i(K_{n_1,n_2,\dots,n_k})=1$. Motivated by Proposition \ref{indep2}, we will consider all regular graphs $G$ with $0<\mu_2(G)\leq 1$ and $i(G)=2$. Since $A(G)+A(\bar{G})=J-I$, it follows that $0<\mu_2(G)\leq 1$ is equivalent to $-2\leq\mu_n(\bar{G})<-1$. Moreover, any eigenvector of $A(\bar{G})$ associated to $\mu_n(\bar{G})$ is an eigenvector of $A(G)$ associated to $\mu_2(G)$, vice versa. It follows that the eigenspace associated to $\alpha(G)$ coincides with the eigenspace associated to $\mu_n(\bar{G})$, which is the least adjacency eigenvalue of $\bar{G}$. Furthermore, the eigenspace corresponding to $\mu_n(\bar{G})$ is the same as the eigenspace corresponding to $\lambda_1(\bar{G})$. Recall that $i_{\lambda}^*(G):=\mathrm{min}\{i_{\lambda}(\mathbf{x})|A(G)\mathbf{x}={\lambda}\mathbf{x}\}$. Therefore, for a regular graph $G$, $$i(G)=i_{\mu_2}^*(G)=i_{\mu_n}^*(\bar{G})=i_{\lambda_1}(\bar{G}).$$

Let $G$ be a connected regular graph of order $n$ with $i(G)=2$. Then  $i_{\mu_n}^*(\bar{G})=2$. It can be easily checked that $G$ is connected if and only if $\bar{G}$ is not expressed as a join of graphs. Hence, the difference between the degree in $\bar{G}$ and $\mu_n(\bar{G})$, which is the largest Laplacian eigenvalue of $\bar{G}$, is less than $n$. Suppose that $\bar{G}$ is disconnected and $H_j$ is a component on $m_j$ vertices in $\bar{G}$ for $j=1,\dots,k$ for some $k\geq 2$. Then there exist components $H_{j_1},\dots,H_{j_q}$ for some $1\leq q\leq k$ such that $\mu_n(\bar{G})=\mu_{m_{j_i}}(H_{j_i})$ for $i=1,\dots,q$. It follows that $$i_{\mu_{m_{j_i}}}^*(H_{j_i})\geq i_{\mu_n}^*(\bar{G})$$ for $i=1,\dots,q$. Since the eigenspace of $\bar{G}$ corresponding to $\mu_n$ is the direct sum of the eigenspaces associated to $\mu_{m_{j_i}}$ of $H_{j_i}$ for $i=1,\dots,q$, the condition $i_{\mu_n}^*(\bar{G})=2$ implies that there exists an $i\in\{1,\dots,q\}$ such that $i_{\mu_{m_{j_i}}}^*(H_{j_i})=2$. Thus, we have the following result.

\begin{lemma}\label{lemma0:i(G)=2}
	Let $G$ be a connected regular graph of order $n$. Suppose that $H_j$ is a component on $m_j$ vertices in $\bar{G}$ for $j=1,\dots,k$ for some $k\geq 1$. We have $i(G)=2$ if and only if there exists a component $H_j$ for $j\in\{1,\dots,k\}$ such that $\mu_{m_i}(H_i)\geq\mu_{m_j}(H_j)$ for all $1\leq i\leq k$ and $i_{\mu_{m_j}}^*(H_j)=2$.
\end{lemma}

Lemma \ref{lemma0:i(G)=2} tells us that to understand a regular graph $G$ with $i(G)=2$, we should investigate the components of the complement of $G$. Specifically, we may narrow our focus to eigenvectors of the least adjacency eigenvalue $-2\leq\mu_n<-1$ of a connected $r$--regular graph $H$ of order $n$ where $r-\mu_n<n$, that is, $H$ can not be written as a join of graphs.

It appears in \cite{book:line} that an $r$--regular graph $H$ of order $n$ with $\mu_n(H)\geq -2$ is either a line graph, a cocktail party graph or a regular exceptional graph. It is known that every cocktail party graph is written as a join of graphs. So, all cocktail party graphs are excluded.

\begin{proposition}\cite{book:line}\label{mu>-2}
	A connected regular graph with least adjacency eigenvalue greater than $-2$ is either a complete graph or an odd cycle.
\end{proposition}

Since $i(K_n)=1$, $K_n$ is ruled out. We will consider eigenvectors of the least adjacency eigenvalue of a cycle $C_n$ of length $n$. As stated in \cite{Brower:spectral}, for $\ell=0,\dots, n-1$, $2\mathrm{cos}(\frac{2\pi\ell}{n})$ is an eigenvalue of $A(C_n)$ associated to $\mathbf{x}_\ell=\left[1,\epsilon^\ell,\dots,\epsilon^{(n-1)\ell}\right]^T$ where $\epsilon=e^{\frac{2\pi i}{n}}$. If $n$ is even, then $\mu_n(C_n)$ is simple and $\mathbf{x}_{\frac{n}{2}}=\left[1,-1,1,\dots,1,-1\right]^T$ is a corresponding eigenvector. So, we have $i_{\mu_n}^*(C_n)=\frac{n}{2}$ for even $n$. Suppose that $n$ is odd. Then, the algebraic multiplicity of $\mu_n$ is $2$, and corresponding linearly independent eigenvectors are $\mathbf{x}_{\frac{n-1}{2}}$ and $\mathbf{x}_{\frac{n+1}{2}}$. Let $\mathbf{v}=[v_0,\dots,v_{n-1}]^T$ and $\mathbf{w}=[w_0,\dots,w_{n-1}]^T$ where $v_j=(-1)^j\mathrm{cos}(\frac{\pi}{n}j)$ and $w_j=(-1)^j\mathrm{sin}(\frac{\pi}{n}j)$ for $j=0,\dots, n-1$, respectively. One can verify that $\mathbf{v}=\frac{\mathbf{x}_{\frac{n-1}{2}}+\mathbf{x}_{\frac{n+1}{2}}}{2} \text{ and } \mathbf{w}=\frac{-\mathbf{x}_{\frac{n-1}{2}}+\mathbf{x}_{\frac{n+1}{2}}}{2i}$. Hence, in order to find $i_{\mu_n}^*(C_n)$ for odd $n$, we need to consider all possible linear combinations of $\mathbf{v}$ and $\mathbf{w}$.

\begin{proposition}
	Let $C_n$ be a cycle of length $n$. Then, $i_{\mu_n}^*(C_n)=\lfloor\frac{n}{2}\rfloor$.
\end{proposition}
\begin{proof}
	For an even cycle, it is clear that $i_{\mu_n}^*(C_n)=\frac{n}{2}$. Suppose that $n$ is odd. Since every Fiedler vector of $C_n$ is a linear combination of $\mathbf{v}$ and $\mathbf{w}$, 
	$$i_{\mu_n}^*(C_n)=\mathrm{min}\{i_{\mu_n}(c_1\mathbf{v}+c_2\mathbf{w})|c_1,c_2\in\mathbb{R},\;c_1c_2\neq 0\}.$$
	Let $\mathbf{u}=c_1\mathbf{v}+c_2\mathbf{w}$ where $\mathbf{u}=\left[u_0,\dots,u_{n-1}\right]^T$. If $c_1=0$ and $c_2\neq 0$, then $i_{\mu_n}^*(\mathbf{u})=\frac{n-1}{2}$. Assume that $c_1\neq 0$. Note that for $j=0,\dots,n-1$, $u_j=c_1v_j+c_2w_j=(-1)^j\sqrt{c_1^2+c_2^2}\mathrm{cos}\left(\frac{\pi}{n}j-\theta\right)$ where $\mathrm{tan}(\theta)=\frac{c_2}{c_1}$. We have $u_j u_{j+1}=-(c_1^2+c_2^2)\mathrm{cos}\left(\alpha_j\right)\mathrm{cos}\left(\alpha_j+\frac{\pi}{n}\right)$ where $\alpha_j=\frac{\pi}{n}j-\theta$. One can check that $u_ju_{j+1}>0$ if and only if $\alpha_j\in(0,\frac{\pi}{2})$ and $\alpha_j+\frac{\pi}{n}\in(\frac{\pi}{2},\pi)$, or $\alpha_j\in(\pi,\frac{3\pi}{2})$ and $\alpha_j+\frac{\pi}{n}\in(\frac{3\pi}{2},2\pi)$. Suppose that $u_j\neq 0$ for all $j=0,\dots,n-1$. Since $\alpha_0,\dots,\alpha_{n-1}\in [-\theta,-\theta+\pi)$, there exists at most one index $j$ in $\{0,\dots,n-2\}$ such that $u_ju_{j+1}>0$. Hence, since $u_ju_{j+1}>0$ implies that $u_j$ and $u_{j+1}$ have the same sign, a change of signs between $u_j$ and $u_{j+1}$ for $j=0,\dots,n-2$ occurs at least $(n-2)$ times. It follows that there are either $\frac{n-1}{2}$ negative and $\frac{n+1}{2}$ positive signs in $\mathbf{u}$ or $\frac{n-1}{2}$ positive and $\frac{n+1}{2}$ negative signs in $\mathbf{u}$. Therefore, $i_{\mu_n}^*(\mathbf{u})=\frac{n-1}{2}$. Assume that there exists $j_0\in\{0,\dots, n-1\}$ such that $u_j=0$. Since $\alpha_0,\dots,\alpha_{n-1}\in [-\theta,-\theta+\pi)$, the $j_0$ is the only solution to $u_j=0$ for $j=0,\dots,n-1$. Consider $u_{j_0-1}u_{j_0+1}=(c_1^2+c_2^2)\mathrm{cos}\left(\frac{(j_0-1)\pi}{n}-\theta\right)\mathrm{cos}\left(\frac{(j_0+1)\pi}{n}j-\theta\right)$. Since $\frac{(j_0-1)\pi}{n}\in(0,\frac{\pi}{2})$ and $\frac{(j_0+1)\pi}{n}\in(\frac{\pi}{2},\pi)$, or $\frac{(j_0-1)\pi}{n}\in(\pi,\frac{3\pi}{2})$ and $\frac{(j_0+1)\pi}{n}\in(\frac{3\pi}{2},2\pi)$, we obtain $u_{j_0-1}u_{j_0+1}<0$. Furthermore, $u_ju_{j+1}<0$ for $j\in\{0,\dots,n-2\}\backslash\{j_0-1,j_0\}$. Then, there are $\frac{n-1}{2}$ positive and negative signs, respectively, and one $0$ in $\mathbf{u}$. Hence, $i_{\mu_n}^*(\mathbf{u})=\frac{n-1}{2}$. Therefore, we have the desired result.
\end{proof}

\begin{corollary}\label{cor:C5}
	Let $C_n$ be a cycle of length $n$. Then, $i_{\mu_n}^*(C_n)=2$ if and only if $n=4,5$.
\end{corollary}

\begin{lemma}\label{lemma1:i(G)=2}
	Suppose that a connected regular graph $H$ of order $n$ has $\mu_n(H)>-2$. Then, $i_{\mu_n}^*(H)=2$ if and only if $H=C_5$.
\end{lemma}
\begin{proof}
	It is immediately proved by Proposition \ref{mu>-2} and Corollary \ref{cor:C5}.
\end{proof}

Let $\be_i$ be a vector whose $i^\text{th}$ component is $1$ and zeros elsewhere. The size is clear from the text.

\begin{definition}\cite{book:line}
	For $n>1$, let $D_n$ be the set of vectors of the form $\pm\be_i\pm\be_j$ ($i<j$).
\end{definition}
\begin{definition}\cite{book:line}
	Let $E_8$ be the set of vectors in $\mathbb{R}^8$ consisting of the $112$ vectors in $D_8$ together with the $128$ vectors of the form $\pm\frac{1}{2}\be_1\pm\frac{1}{2}\be_2\pm\cdots\pm\frac{1}{2}\be_8$, where the number of positive coefficients is even.
\end{definition}

Now, the regular line graphs and regular exceptional graphs with least adjacency eigenvalue $-2$ are left to consider. These graphs are studied in \cite{book:line} using $D_n$ and $E_8$, the so--called, \textit{root systems}. Let $H$ be a graph on $n$ vertices with  least adjacency eigenvalue $-2$. The symmetric matrix $2I+A(H)$ is positive semi--definite of rank $s$, say. Since $2I+A(H)$ is orthogonally diagonalizable, it follows that $C^TC=2I+A(H)$ where $C$ is an $s\times n$ matrix of rank $s$. According to \cite{book:line}, the column vectors of $C$ are determined by $D_n$ or $E_8$. 

\begin{lemma}\label{root}
	Let $H$ be a connected regular graph with the least adjacency eigenvalue $-2$. If $H$ contains an induced $4$--cycle, there exists an eigenvector $\bx^T=\left[1,-1,1,-1,0,\dots,0\right]$ of $A(H)$ associated with $-2$. 
\end{lemma}
\begin{proof}
	Considering the root systems, there exists a real matrix $C$ such that $C^TC=2I+A(H)$. Since $H$ contains an induced $4$--cycle, without loss of generality, the leading principal $4\times 4$ submatrix of $A(H)$ is an adjacency matrix of $C_4$. Let the first four columns of $C$ comprise the matrix $\tilde{C}$. Then,
	$$\tilde{C}^T\tilde{C}=2I+A(C_4).$$
	Since $\tilde{\bx}^T=\left[1,-1,1,-1\right]$ is an eigenvector of $A(C_4)$ associated to $-2$, we have that $(\tilde{C}\tilde{\bx})^T\tilde{C}\tilde{\bx}=0$. $C$ is real, so $\tilde{C}\tilde{\bx}=0$. Suppose that $\bx^T=\left[1,-1,1,-1,0,\dots,0\right]$. Then,  $C\bx=0$. Therefore, follows that $\bx$ is an eigenvector of $A(H)$ associated to $-2$.
\end{proof}

\begin{lemma}\label{lemma2:i(G)=2}
	Let $H$ be a connected $r$--regular graph of order $n$ with $\mu_n(H)=-2$ where $r+2<n$. Then, $i_{\mu_n}^*(H)=2$ if and only if $H$ contains a $4$--cycle as an induced subgraph.
\end{lemma}
\begin{proof}
	Suppose that $i_{\mu_n}^*(H)=2$. Since $r+2<n$, the complement $\bar{H}$ of $H$ is connected and regular with $\mu_2(\bar{H})=1$. Moreover, $i_{\mu_2}(\bar{H})=i(\bar{H})=2$. By Proposition \ref{indep2}, $\bar{H}$ contains two non--adjacent edges as an induced subgraph. Therefore, $H$ has an induced subgraph $C_4$.
	
	Conversely, by Lemma \ref{root}, there exists an eigenvector $\bx^T=\left[1,-1,1,-1,0,\dots,0\right]$ of $A(H)$ associated to $-2$. So, $i_{\mu_n}^*(H)\leq2$. Since $\mu_n\neq r$, any eigenvector associated to $\mu_n$ must contain negative and positive components. So, $i_{\mu_n}^*(H)>0$. Suppose that $i_{\mu_n}^*(H)=1$. Since $\bar{H}$ is connected, it follows that $i_{\mu_n}^*(H)=i_{\mu_2}(\bar{H})=i(\bar{H})=1$. So, $\bar{H}$ can be expressed as a join of two graphs by Theorem \ref{equiv:i(G)=1}. This is a contradiction to being a connected graph. Therefore, $i_{\mu_n}^*(H)=2$.
\end{proof}

Here is the our main result in this section regarding the characterization of all connected regular graphs $G$ with $i(G)=2$.

\begin{theorem}\label{Theorem:i(G)=2 for regular}
	Let $G$ be a connected $r$--regular graph of order $n$. Then, $i(G)=2$ if and only if there exists a component $H$ of order $m$ in $\bar{G}$ such that $\mu_n(\bar{G})=\mu_m(H)=\alpha(G)-r-1$ and $H$ satisfies either
	\begin{enumerate}
		\item $r-1<\alpha(G)<r$ and $H=C_5$, or
		\item $\alpha(G)=r-1$, $H$ is not a cocktail party graph and $H$ contains $C_4$ as an induced subgraph.
	\end{enumerate}
\end{theorem}
\begin{proof}
	Combining Lemmas \ref{lemma0:i(G)=2}, \ref{lemma1:i(G)=2} and \ref{lemma2:i(G)=2}, we obtain the desired result.
\end{proof}

\begin{example}{\rm{
	Let $H$ be a strongly regular graph with least adjacency eigenvalue $-2$. According to Seidel's classification \cite{Seidel:strong with $-2$}, $H$ is one of
	\begin{enumerate}
		\item the complete $n$--partite graph $K_{2,\dots,2}$ for $n\geq 2$,
		\item the Petersen graph,
		\item the line graph of $K_n$ for $n\geq 5$,
		\item the Cartesian product of two $K_n$s for $n\geq 3$,
		\item the Shrikhande graph,
		\item one of the three Chang graphs,
		\item the Clebsch graph,
		\item the Schl\"{a}fli graph.
	\end{enumerate}
	Note that $K_{2,\dots,2}$ is expressed as a join of graphs. The girth of the Petersen graph is $5$. It can be checked that $H$ has an induced $4$--cycle if and only if the line graph of $H$ contains $C_4$ as an induced graph. This implies that any line graph of a complete graph is $C_4$--free. For the other graphs from (4) to (8), it can be checked that they have $C_4$ as an induced subgraph. Therefore, if a connected regular graph $G$ has one of graphs from (4) to (8) as a component in $\bar{G}$, then $i(G)=2$.}}
\end{example}

		

{\small
{\em Authors' addresses}:
{\em Sooyeong Kim}, University of Manitoba, Winnipeg, Canada
 e-mail: \texttt{kims3428@\allowbreak myumanitoba.ca}, {\em Steve Kirkland}, University of Manitoba, Winnipeg, Canada
 e-mail: \texttt{stephen.kirkland@\allowbreak umanitoba.ca}.

}


{\small
\begin{thebibliography}{999}

\bibitem{Brower:spectral}
{\it Brouwer, Andries E. and Haemers, Willem H.}:
Spectra of graphs.
Springer, New York (2012), xiv+250. Zbl 1231.05001, MR2882891, DOI 10.1007/978--1--4614--1939--6

\bibitem{book:line}
{\it Cvetkovi\'{c}, Drago\v{s} and Rowlinson, Peter and Simi\'{c}, Slobodan}:
Spectral generalizations of line graphs.
Cambridge University Press {\bf 314} (2004), xii+298. Zbl 1061.05057, MR2120511, DOI 10.1017/cbo9780511751752

\bibitem{Cvetkovic:second}
{\it Cvetkovi\'{c}, Drago\v{s} and Simi\'{c}, Slobodan}:
The second largest eigenvalue of a graph (a survey).
Filomat {\bf 9, part 3} (1995), 449--472. Zbl 0851.05078, MR1385931

\bibitem{Fiedler:symmetric}
{\it Fiedler, Miroslav}:
A property of eigenvectors of nonnegative symmetric matrices and its application to graph theory.
Czechoslovak Math. J. {\bf 25(100)} (1975), 619--633. Zbl 0437.15004, MR0387321

\bibitem{Fiedler:algebraic}
{\it Fiedler, Miroslav}:
Algebraic connectivity of graphs.
Czechoslovak Math. J. {\bf 23(98)} (1973), 298--305. Zbl 0265.05119, MR0318007

\bibitem{Steve:connectivity}
{\it Kirkland, Stephen J. and Molitierno, Jason J. and Neumann, Michael and Shader, Bryan L.}:
On graphs with equal algebraic and vertex connectivity.
Linear Algebra Appl. {\bf 341} (2002), 45--56. Zbl 0991.05071, MR1873608, DOI 10.1016/s0024--3795(01)00312--3

\bibitem{Merris:degree maximal graph}
{\it Merris, Russell}:
Degree maximal graphs are {L}aplacian integral.
Linear Algebra Appl. {\bf 199} (1994), 381--389. Zbl 0795.05091, MR1274427, DOI 10.1016/0024--3795(94)90361--1

\bibitem{Merris:join and spect}
{\it Merris, Russell}:
Laplacian graph eigenvectors.
Linear Algebra Appl. {\bf 278} (1998), 221--236. Zbl 0932.05057, MR1637359, DOI 10.1016/s0024--3795(97)10080--5


\bibitem{Seidel:strong with $-2$}
{\it Seidel, J. J.}:
Strongly regular graphs with {$(-1,\,1,\,0)$} adjacency matrix having eigenvalue {$3$}.
Linear Algebra Appl. {\bf 1} (1968), 281--298. Zbl 0159.25403, MR234861, DOI 10.1016/0024--3795(68)90008--6

\bibitem{Urchel:maximal}
{\it Urschel, John C. and Zikatanov, Ludmil T.}:
On the maximal error of spectral approximation of graph
bisection.
Linear Multilinear Algebra. {\bf 64} (2016), 1972--1979. Zbl 1352.05120, MR3521152

\bibitem{Urschel:bisection}
{\it Urschel, John C. and Zikatanov, Ludmil T.}:
Spectral bisection of graphs and connectedness.
Linear Algebra Appl. {\bf 449} (2014), 1--16. Zbl 1286.05101, MR3191855, DOI 10.1016/j.laa.2014.02.007

\end{thebibliography}}
\end{document}